\pdfoutput=1 
\documentclass[dvipsnames,table]{amsart}

\usepackage[utf8]{inputenc}
\usepackage[T1]{fontenc}
\usepackage{geometry}
\usepackage{xcolor}

\usepackage{amsfonts}
\usepackage{amssymb}
\usepackage{mathtools}
\usepackage{tikz}
\usetikzlibrary{calc,tikzmark}
\usepackage{tikz-cd}
\usepackage{pgfplots}
\pgfplotsset{compat=1.18}
\usepackage{float}
\usepackage{booktabs}

\usepackage[backend=biber,style=alphabetic,sorting=nyt,citestyle=alphabetic,maxbibnames=99,maxalphanames=4]{biblatex}
\usepackage[
    pdfusetitle,
    colorlinks=true,
    linkcolor=CtpMauve,
    citecolor=CtpBlue,
    urlcolor=CtpPeach,
    filecolor=CtpGreen,
    menucolor=CtpRed,
    runcolor=CtpRed,
    bookmarks=true,
    bookmarksopen=true,
    bookmarksopenlevel=2,
    pdfstartview=Fit,
    pdfpagemode=UseOutlines
]{hyperref}
\usepackage{algorithm}
\usepackage{algpseudocode}

\usepackage[nameinlink,capitalise,noabbrev]{cleveref}
\crefname{data}{data}{data}
\Crefname{data}{Data}{Data}

\newtheorem{theorem}{Theorem}[section]
\newtheorem{proposition}[theorem]{Proposition}

\theoremstyle{definition}
\newtheorem{example}[theorem]{Example}
\newtheorem{data}[theorem]{Data}
\theoremstyle{remark}

\newtheorem*{remark*}{Remark}
\numberwithin{equation}{section}
\definecolor{CtpPink}{HTML}{ea76cb}
\definecolor{CtpMauve}{HTML}{8839ef}
\definecolor{CtpRed}{HTML}{d20f39}
\definecolor{CtpPeach}{HTML}{fe640b}
\definecolor{CtpYellow}{HTML}{df8e1d}
\definecolor{CtpGreen}{HTML}{40a02b}
\definecolor{CtpTeal}{HTML}{179299}
\definecolor{CtpSky}{HTML}{04a5e5}
\definecolor{CtpBlue}{HTML}{1e66f5}
\definecolor{CtpText}{HTML}{000000}
\definecolor{CtpOverlay2}{HTML}{767676}
\definecolor{CtpOverlay0}{HTML}{9e9e9e}
\definecolor{CtpSurface1}{HTML}{c8c8c8}
\definecolor{CtpSurface0}{HTML}{dcdcdc}
\definecolor{CtpBase}{HTML}{ffffff}
\definecolor{CtpMantle}{HTML}{f5f5f5}

\definecolor{darkcyan}{rgb}{0, 0.7, 0.7}
\definecolor{darkgreen}{rgb}{0, 0.7, 0}
\definecolor{truemagenta}{rgb}{1, 0, 1}
\definecolor{amber}{rgb}{1.0, 0.75, 0.0}

\tikzset{
  panel/.style={
    fill=CtpMantle,
    rounded corners,
    inner sep=8pt,
    text width=0.95\textwidth,
    align=center,
  },
  ctpnode/.style={
    draw=CtpText,
    rectangle,
    rounded corners,
    text=CtpBase,
    font=\small,
    minimum width=32pt,
    minimum height=14pt,
    inner sep=2pt,
  },
  description/.style={
    fill=CtpSurface0,
    text=CtpText,
    rounded corners=3pt,
    draw=CtpOverlay0,
    inner sep=5pt,
    text width=5cm,
    align=center,
    font=\footnotesize,
  },
  description-wide/.style={
    description,
    text width=5.5cm,
  },
  description-small/.style={
    description,
    text width=2.3cm,
  },
  description-leibniz/.style={
    description,
    text width=6.6cm,
  },
}
\usepackage{xspace}
\newcommand{\update}{\mathrel{\gets}}
\newcommand{\flowchartscale}{0.8}

\title{Automated proofs of unstable Adams differentials}
\author{Jake Francis Baer}
\thanks{Email: \texttt{francis@wayne.edu}}

\begin{document}

\begin{abstract}

\noindent We present a computer-based approach to computing 
differentials in the unstable Adams spectral sequence by systematically
applying the unstable Leibniz rule and naturality with respect to maps in the EHP sequence. We record our results in tables of upper and lower bounds on the orders of 2-primary unstable homotopy groups of spheres through the unstable 50-stem. We provide examples of proofs for several differentials and give a guide to interpreting the associated unstable Adams charts and flow chart diagrams for differential proofs. \end{abstract}

\maketitle

\setcounter{tocdepth}{1}
\tableofcontents

\section{Introduction}
\noindent The homotopy type of a cell complex is encoded in its attaching maps, which specify how each cell is glued to the lower-dimensional skeleton during inductive construction. Cellular approximation allows such maps to be decomposed into elements of homotopy groups of spheres, so the classification of cell complexes up to homotopy equivalence is governed by the structure of these groups.
The explicit computation of unstable homotopy groups of spheres begins with Hopf's construction of the first nontrivial elements via the Hopf invariant~\cite{Hopf1931} and the low-stem computations of Pontryagin~\cite{Pontrjagin_1938} and G.~W. Whitehead~\cite{Whitehead_GW_1950}; Serre's spectral sequence methods made systematic calculation possible~\cite{Serre1953}. Toda’s foundational work applied the EHP sequence and bracket arguments, yielding tables of $\pi_{n+k}(S^n)$ for $k\leq 19$~\cite{Toda_1962}. Mimura and collaborators pushed these computations through roughly the $24$-stem~\cite{MimuraToda1963, Mimura1964, Mimura1965, MimuraMoriOda1975}. Oda extended the $2$-primary computations through the 30-stem with partial results in the 31- and 32-stems using refinements of Toda’s methods~\cite{Oda1976, Oda1977}. The remaining ambiguities were resolved by Inoue--Mukai and Miyauchi--Mukai~\cite{InoueMukai2012, MiyauchiMukai2017}. Separately, Curtis and Mahowald determined the unstable homotopy groups of $S^3$ through the 52-stem in the only published application of the unstable Adams spectral sequence to computing unstable stems~\cite{Curtis_Mahowald_1989}.

This paper consists of two distinct computations.
The first is a modern reimplementation of Curtis' algorithm for computing the homology of the lambda algebra, which is the $E_2$-page of the unstable Adams spectral sequence. We compute the unstable Adams $E_2$-page for all spheres through total degree 76 as well as all values of the algebraic EHP maps (\Cref{prop:ehp}) and all unstable algebraic compositions on the Adams $E_2$-page. The second computation uses the algebraic structure of the unstable Adams $E_2$-page and input differentials from the stable Adams spectral sequence to propagate unstable Adams differentials using the unstable Leibniz rule and naturality of the EHP maps. This propagator tries all possible combinations of EHP maps and unstable compositions to constrain the possible values of an unstable Adams differential until only one remains. The precise combination of EHP maps and unstable compositions used to compute a differential is recorded in computer-generated flow charts. By resolving unstable Adams differentials and computing the $E_\infty$-page we prove the following main theorem. 

\begin{theorem}
\label{thm:main}
Upper and lower bounds on the orders of the $2$-primary homotopy groups $\pi_{n+k}(S^n)$ for $0 \le k \le 50$ are given in \Crefrange{tab:bounds-true-2-14}{tab:bounds-true-39-50}.
\end{theorem}
\noindent Through the unstable 32-stem these bounds are consistent with, and largely recover, the previously known orders due to the authors cited above (together with Curtis--Mahowald's computations for $S^3$); the uniform machine-generated determination of unstable Adams differentials for all spheres, and the resulting bounds beyond the 32-stem, are new.

\subsection{The lambda algebra}

In \Cref{sec:lambda} we describe our computational approach to the structure of the unstable Adams $E_2$-page. We reimplement Curtis' lambda algebra algorithm from scratch. Our implementation contains nothing conceptually new and the theory underlying this algorithm is well exposited in the many sources we reference in \Cref{sec:lambda}. By writing our implementation in a fast modern language and optimizing the key inner loops, we compute the homology of the lambda algebra and all algebraic compositions and algebraic EHP maps through total degree 76. While this algorithm for computing algebraic EHP maps and algebraic compositions has existed for several decades, neither of these has been computed exhaustively in a significant range, let alone applied to the computation of unstable Adams differentials. The most conceptually original part of \Cref{sec:lambda} consists of our generalization and exhaustive computation of the map of $E_2$-pages used by Curtis and Mahowald to compute unstable Adams differentials for $S^3$ by comparison with the stable Adams spectral sequence for $\mathbb{S}/2$~\cite{Curtis_Mahowald_1989}. While Curtis and Mahowald use this map to compute unstable Adams differentials, they never compute its values explicitly and only indirectly deduce its values by using the fact that this map is an isomorphism in sufficiently high filtration relative to the stem. This map is essential to our results since without it there are many differentials for low-dimensional spheres like $S^3$ and $S^5$ that we are unable to compute using EHP naturality and the Leibniz rule alone.
\begin{remark*}
  By running our Curtis algorithm on the Wayne State University high-performance grid, we computed the homology of the lambda algebra additively in degrees $s + f < 90$. The database of cycle representatives for this computation exceeds 50~GB in size, so we did not attempt to compute algebraic compositions between total degrees 76 and 90.
\end{remark*}
\subsection{Unstable Adams differentials}
In \Cref{sec:ehp} we describe our differential propagation algorithm, which uses the algebraic data from \Cref{sec:lambda}. Our algorithm systematically applies naturality of unstable Adams differentials with respect to EHP maps and the unstable Leibniz rule, which are stated precisely in \Cref{prop:ehp} and \Cref{prop:leibniz} below. Our approach to differential propagation is based on the approach in~\cite{Beauvais_Feisthauer_2022} to propagating stable Adams $d_2$-differentials using the Leibniz rule for the sphere spectrum. Our implementation starts with the space of all possible matrix representations of a $d_r$-differential considered as a map of $\mathbb{F}_2$-vector spaces. In order to apply naturality and the Leibniz rule, we rephrase \Cref{prop:ehp} and \Cref{prop:leibniz} as statements that certain diagrams have to commute. Our propagation algorithm then iteratively checks for matrices that fail to commute in the relevant diagrams, and removes these from our space of possible $d_r$-differential values. By repeating this process many times on all bidegrees for all spheres, we can narrow down the space of possible $d_r$-differentials until there is a unique matrix that respects all of the constraints imposed by \Cref{prop:ehp} and \Cref{prop:leibniz}. This process of iteratively applying constraints in order to deduce a unique $d_r$-differential value is recorded in the form of flow charts from which the reader may reconstruct the proof of any differential. 
\subsection{Results}
In \Cref{sec:results} our results are presented in unstable Adams charts and tables of upper and lower bounds on the 2-primary homotopy groups of spheres through dimension 50. The algebraic data generated in \Cref{sec:lambda} and the differential propagation algorithm described in \Cref{sec:ehp} allow us to determine all differentials in the unstable Adams spectral sequence for most spheres of dimension at most 50 and in most stems up to 50. We have partial results in some cases. Our results are visualized in interactive Adams charts, rendered with the SeqSee visualization engine~\cite{Beauvais_Feisthauer_Isaksen_2025}, displaying the
spectral sequences for spheres $S^n$ with $2 \leq n \leq 50$. These charts display proven differentials as well as the remaining possible unknown differentials on each page. The unstable Adams spectral sequence gives us an upper and lower bound on the order of each $2$-primary unstable stem. These upper and lower bounds are tabulated in \Cref{app:bounds}. The homotopy groups whose orders are given in \Cref{tab:bounds-true-27-38} and \Cref{tab:bounds-true-39-50} occur entirely in the metastable range, i.e.\ the range $s \leq 3n - 4$. The most interesting entries are therefore located in the purely unstable region in the lower left corner of \Cref{tab:bounds-true-2-14}. For some spheres such as $S^5$, $S^9$, $S^{11}$, $S^{13}$, $S^{17}$, $S^{18}$, and $S^{19}$, few uncertainties exist and we may read off the exact orders of their 2-primary homotopy groups through at least the 45-stem. For other spheres like $S^{12}$ and $S^{24}$ there are many remaining uncertainties and we can only bound the order of most of their homotopy groups instead of providing an exact value. 
\subsubsection{Reading the tables}
\Crefrange{tab:bounds-true-2-14}{tab:bounds-true-39-50} in \Cref{app:bounds} display the upper and lower bounds on the orders of 2-primary unstable homotopy groups computed using the differential propagation algorithm in \Cref{sec:ehp} and the unstable Adams $E_2$-page data from \Cref{sec:lambda}. To suppress redundant stable information, we indicate the stable range by the vertical lines in the upper part of each table. The 2-primary metastable range is contained between the stable range vertical lines and the pink staircase boundary drawn through \Cref{tab:bounds-true-2-14} and \Cref{tab:bounds-true-15-26}. The entry in column $n$ and row $s$ bounds the order of the $2$-primary component of $\pi_{n+s}(S^n)$, recorded through its base-$2$ logarithm. A single entry $e$ marks a homotopy group with no remaining unknown differentials, whose $2$-primary component has order exactly $2^e$. When a homotopy group still admits unknown differentials, we can only bracket
its order. The entry is then a range $e/\ell$, written upper bound first,
meaning that the $2$-primary component of $\pi_{n+s}(S^n)$ has order $2^m$ for
some $\ell \le m \le e$.
\subsection{Relation to prior work}
Our results have been checked against the existing work of Toda, Mimura, Miyauchi, Mukai, Inoue, Mori, and Oda on 2-primary homotopy groups of spheres. Our results are compatible with theirs in the sense that the orders they compute fit between our upper and lower bounds through the unstable 32-stem. There are a few places between the 19-stem and the 32-stem where our computation admits an uncertain differential that could be resolved by referring to the work of these authors. In such places, we choose not to fill in this outside information. By doing this we ensure that our results are entirely machine generated in an attempt to minimize the risk of human error.
In forthcoming joint work with Balderrama, Belmont, and Isaksen, we use the unstable Adams spectral sequence in conjunction with the EHP spectral sequence to compute all metastable homotopy groups of spheres whose stabilizations lie in stems 0 through 66~\cite{BBBI_metastable}.
\section{The unstable Adams \texorpdfstring{$E_2$}{E2}-page}
\label{sec:lambda}

\noindent We begin by reviewing the structure of the unstable Adams spectral sequence and describing the computational methods used to determine its $E_2$-page. Recall that there is a cofree resolution of $\tilde{H}_{\ast}(S^n; \mathbb{F}_2)$ in unstable comodules over the dual Steenrod algebra $\mathcal{A}_*$ that is of finite type in each degree and is presented as $\mathcal{A}_* \underline{\otimes} \Sigma^n \Lambda(n)$, following Singer~\cite{Singer_1975}. Here $\mathcal{A}_* \underline{\otimes} N = \bigoplus_j J(j) \otimes N_j$ is the injective unstable comodule on a graded $\mathbb{F}_2$-vector space $N$, where $J(j) \subseteq \mathcal{A}_*$ is the span of the monomials $\xi_1^{r_1} \xi_2^{r_2} \cdots$ with $\sum_i r_i \leq j$,
and $\Lambda(n)$ is a subcomplex of the lambda algebra $\Lambda$, the differential bigraded algebra presented by
\begin{equation}
  \label{eq:lambda}
  \Lambda = \mathbb{F}_2\langle \lambda_0, \lambda_1, \lambda_2, \dots\rangle / \mathord{\sim}
  \qquad
  \begin{aligned}
    \lambda_i \lambda_{2i+1+m}
      &\sim \textstyle\sum_{j\ge0} \tbinom{m-1-j}{j} \lambda_{i+m-j} \lambda_{2i+1+j} \\
    d(\lambda_n)
      &= \textstyle\sum_{j\ge0} \tbinom{n-1-j}{j+1} \lambda_{n-1-j} \lambda_j 
  \end{aligned}
\end{equation}
with the convention that $\tbinom{a}{b} = 0$ unless $0 \le b \le a$. As an $\mathbb{F}_2$-vector space, $\Lambda$ has a basis of \emph{admissible} monomials $\lambda_{i_1}\lambda_{i_2}\cdots\lambda_{i_f}$, those satisfying $2 i_k \geq i_{k+1}$ for all $k$. We write $\Lambda(n) \subseteq \Lambda$ for the subcomplex spanned by the admissible monomials with $i_1 \leq n - 1$, and $\Lambda^{s,f}(n)$ for its subspace spanned by such monomials of length $f$ whose indices sum to $s$~\cite{Singer_1975}. We write $\alpha \smallsmile \beta$ for the concatenation product of $\alpha, \beta \in \Lambda(n)$, and reserve $\circ$ for the induced composition in homology. We write $[\lambda_I]$ for the class in homology whose cycle representative has leading term $\lambda_I$, with respect to the lexicographic order on the index sequences of admissible monomials.

This algebraic resolution can be realized as a filtration in spaces that begets an unstable version of the Adams spectral sequence and has signature
\begin{align*}
  E_{2}^{s, f}(S^n) = H^{s, f}\Lambda(n) \Rightarrow (\pi_{s + n} S^n)^{\wedge}_{2}
\end{align*}
by work of Massey--Peterson and Harper--Miller~\cite{Massey_Peterson_1967, Harper_Miller_1989}.
An alternative construction of this spectral sequence is given using the lower central series filtration on simplicial groups~\cite{Bousfield_Curtis_1970, Curtis_1971, Liebowitz_1972}; the identification of $\Lambda(n)$ as the $E_1$-page of the $p$-lower central series spectral sequence appears already in \cite[Section~5.4]{Bousfield_Curtis_Kan_Quillen_Rector_Schlesinger_1966}. For a modern survey of these ideas see~\cite{Behrens_Malin_2024}. Recall that, $2$-locally, James' fibration $S^n \to \Omega S^{n+1} \to \Omega S^{2n+1}$ gives rise to the long exact EHP sequence relating the homotopy groups of neighboring spheres, in which $\mathrm{E}$ is induced by suspension, $\mathrm{H}$ by the James--Hopf invariant, and $\mathrm{P}$ by Whitehead products; see~\cite{Mahowald_Thompson_1995} for a survey. When considered as a family, the unstable Adams spectral sequences for all $S^n$ exhibit the following structure:

\begin{proposition}[\cite{Pavutnitskiy_Wu_2019, Singer_1975, Curtis_1971}] There are chain maps
\label{prop:ehp}
\begin{align*}
  \Lambda^{s, f}(n) &\stackrel{\mathrm{E}}{\longrightarrow} \Lambda^{s, f}(n + 1)
\stackrel{\mathrm{H}}{\longrightarrow}
\Lambda^{s - n, f - 1}(2n + 1)
\stackrel{\mathrm{P}}{\longrightarrow}
\Lambda^{s - 1, f + 1}(n),
\end{align*}
which detect $\mathrm{E}$, $\mathrm{H}$, and $\mathrm{P}$ in $\pi_\ast(-)$ and commute with unstable Adams differentials.
\end{proposition}

\begin{proposition}[\cite{Bousfield_Kan_1973a, Bousfield_Kan_1973b, Ivanov_Mikhailov_Wu_2018}]
\label{prop:leibniz}
There exists an unstable composition pairing
\begin{align*}
  H^{s, f}\Lambda(n) \otimes H^\ast \Lambda(n + s) &\xrightarrow{-\circ -} H^\ast \Lambda(n),\\
\alpha \otimes \beta &\longmapsto \alpha \smallsmile \mathrm{E}^{f}\beta,
\end{align*}
which detects geometric compositions in $\pi_\ast(-)$ and respects an unstable Leibniz rule
\begin{align*}
    d_r(\alpha \circ \mathrm{E} \beta) = d_r(\alpha) \circ \beta + \alpha \circ d_r(\mathrm{E} \beta).
\end{align*}
\end{proposition}

\subsection{The Curtis procedure}
Filtering $\Lambda(n)$ by the subcomplexes $\Lambda(m)$ for $m \leq n$ yields a spectral sequence known as the \emph{algebraic EHP spectral sequence}~\cite{Ravenel_1986, Curtis_Goerss_Mahowald_Milgram_1987}, whose filtration quotients are identified with shifted copies of the $\Lambda(2m + 1)$ by the maps $\mathrm{H}$ of \Cref{prop:ehp}.
The \textit{Curtis procedure} computes the homology of $\Lambda(n)$ by computing the differentials in this spectral sequence~\cite{Whitehead_1970,Curtis_1971}: each admissible monomial either survives as the leading term of a cycle representative of a homology class or is cancelled by a differential.
The Curtis procedure has had numerous well-documented implementations~\cite{Hansen_1974,Tangora_1985,Curtis_Goerss_Mahowald_Milgram_1987,Wang_Xu_2016}. Our implementation is written in Rust (version 1.96)~\cite{ehpreprint_data}.
We have computed the homology of $\Lambda(n)$ additively in all total degrees $s + f < 90$; the dataset released with this paper~\cite{ehpreprint_data} covers all total degrees $s + f \leq 76$. The data generated by our implementation of the Curtis homology procedure consists of the following:
\begin{data}
\label{data:table}
The plain-text files \texttt{rank.csv} and \texttt{names.json} in the \texttt{data/} directory of~\cite{ehpreprint_data} contain
\begin{enumerate}
  \item The dimension of $H^{s, f}\Lambda(n)$ for all $s + f \leq 76$
  \item For each basis element, the lexicographically leading admissible monomial of a cycle representative
\end{enumerate}
The full cycle representatives, together with the minimal set of differentials in the algebraic EHP spectral sequence used to reduce them, are regenerated on demand by our released implementation; at the full range they are too large to ship.
\end{data}
\noindent To operate on a class in $H^{s, f}\Lambda(n)$ we first lift its cycle representative to $\Lambda(n)$ and then operate on it, in our case either by some concatenation product or by the $\mathrm{H}$ map in \Cref{prop:ehp}. We then take the resulting polynomial in $\Lambda(m)$ and systematically add boundaries until we have expressed it as a linear combination of cycle representatives in the homology of $\Lambda(m)$. This approach has been used by Tangora to compute some products and Massey products in the classical Adams $E_2$-page~\cite{Tangora_1993, Tangora_1994}. We use this method to compute the following:
\begin{data}
  \label{data:ehp-comp}
  The plain-text table files \texttt{E.csv}, \texttt{H.csv}, \texttt{P.csv}, and \texttt{compositions.csv} in the \texttt{data/} directory of~\cite{ehpreprint_data} contain a full description through total degree 76 of the values of the maps in \Cref{prop:ehp} and \Cref{prop:leibniz} on every $\mathbb{F}_2$-linear basis element from \Cref{data:table}.
\end{data}

\subsection{The fiber of the double suspension}
Write $e_m \Lambda$ for a copy of $\Lambda$ shifted by a generator $e_m$ in degree $m$, so that $\mathbb{F}_2\{e_m\} = \Sigma^m \mathbb{F}_2$. The $E_2$-page of the Adams spectral sequence for $\Sigma^{2n - 1}\mathbb{S}/2$
can be described by taking the homology of $e_{2n} \Lambda \oplus e_{2n -
  1} \Lambda $ with differential inherited from $\Lambda$ along with $d(e_{2n}) = e_{2n - 1} \lambda_0$~\cite{Priddy_1970}. The following result is implicit in the cited papers, but for completeness we provide a proof sketch along with all relevant references.
\begin{theorem}[\cite{Mahowald_1975, Mahowald_1982, Mahowald_Thompson_1994}]
\label{thm:mah}
For each $n$ there exists the following:
\begin{enumerate}
\item a chain map
\begin{equation*}
\begin{aligned}
  \Lambda(2n + 1) \to e_{2n} \Lambda \oplus e_{2n - 1} \Lambda \qquad\qquad
  \lambda_{2n - 1}\,\lambda_I &\mapsto e_{2n - 1}\,\lambda_I \\
  \lambda_{2n}\,\lambda_I &\mapsto e_{2n}\,\lambda_I + \epsilon\, e_{2n - 1}\,\lambda_{4n + 1}\,\lambda_{I'},
\end{aligned}
\end{equation*}
where $\epsilon = 1$ if $\lambda_I = \lambda_{4n}\,\lambda_{I'}$ and $\epsilon = 0$ otherwise, and where monomials beginning with $\lambda_i$ for $i < 2n - 1$ are sent to $0$;
\item a map $\Omega^{3} S^{2n + 1} \to \Omega^{\infty}\Sigma^{\infty}\Sigma^{4n - 4}\mathbb{RP}^2$ which can be covered by a map of unstable resolutions described on the $E_1$-page by the chain map in (1).
\end{enumerate}
\end{theorem}
\begin{proof}
\textit{Step 1.}
  Write $\operatorname{fib}_n(\mathrm{E}^2)$ for the fiber of the double suspension $S^{2n - 1} \xrightarrow{\mathrm{E}^2} \Omega^2 S^{2n + 1}$, and let $C_n$ denote the mapping cone of the chain-level double suspension $\mathrm{E}^2 \colon \Lambda(2n - 1) \to \Lambda(2n + 1)$ from \Cref{prop:ehp}, so that $C_n = \Lambda^{\ast, \ast + 1}(2n - 1) \oplus \Lambda^{\ast, \ast}(2n + 1)$ as bigraded modules. Following~\cite{Mahowald_Thompson_1987}, we build a resolution for $\operatorname{fib}_n(\mathrm{E}^2)$ whose associated algebraic resolution is $\mathcal{A}_* \underline{\otimes} \Sigma^{2n - 2} C_n$. This resolution is constructed so that the map $\Omega^3 S^{2n + 1} \to \operatorname{fib}_n(\mathrm{E}^2)$ is covered by a map of resolutions that induces the inclusion of the second summand
  \begin{equation}
    \label{eq:cone-map}
    \mathcal{A}_* \underline{\otimes} \Sigma^{2n - 2} \Lambda(2n + 1) \hookrightarrow \mathcal{A}_* \underline{\otimes} \Sigma^{2n - 2} C_n.
  \end{equation}

  Note that $C_n$ is chain homotopy equivalent to the quotient complex $\Lambda(2n + 1)/\Lambda(2n - 1)$, which Mahowald~\cite[Section~2]{Mahowald_1975} presents as
\begin{align*}
  \operatorname{cof}_n(d_1)
  &\coloneqq \kappa_1\, \Lambda(4n + 1) \oplus  \kappa_0\, \Lambda(4n - 1) ,
  &
  d(\kappa_1) &= \kappa_0\, \lambda_0,
  \\
  |\kappa_i| &= (2n - 1 + i,\, 1),
  &
  d(\kappa_1\, \lambda_{4n}) &= \kappa_0\bigl(\lambda_0 \lambda_{4n} + d(\lambda_{4n+1})\bigr),
\end{align*}
with differential otherwise inherited from $\Lambda(4n + 1)$ and $\Lambda(4n - 1)$; here $d(\lambda_{4n + 1})$ is given by \eqref{eq:lambda}. Our notation is meant to reflect that this construction is equal to the mapping cone for the algebraic EHP $d_1$-differential $\Lambda(4n + 1) \xrightarrow{\mathrm{H} \circ \mathrm{P}} \Lambda(4n - 1)$. Composing the inclusion \eqref{eq:cone-map} with this quotient map, we conclude that the map of resolutions covering $\Omega^3 S^{2n + 1} \to \operatorname{fib}_n(\mathrm{E}^2)$ induces the chain map
\begin{equation}
\label{eq:firstmap}
\begin{aligned}
  \Lambda(2n + 1) &\to \operatorname{cof}_n(d_1),
  & \lambda_{2n}\,\alpha &\mapsto \kappa_1\,\alpha,
  & \lambda_{2n - 1}\,\alpha &\mapsto \kappa_0\,\alpha.
\end{aligned}
\end{equation}

\textit{Step 2.}
Cohen showed that there is a map $\operatorname{fib}_n(\mathrm{E}^2) \to \Omega^4\operatorname{fib}_{n+1}(\mathrm{E}^2)$ that has degree $1$ on the bottom cell and such that the colimit of the sequence of maps
    \begin{align*}
      \operatorname{fib}_n(\mathrm{E}^2) \to \Omega^4 \operatorname{fib}_{n+1}(\mathrm{E}^2) \to \Omega^8 \operatorname{fib}_{n+2}(\mathrm{E}^2) \to \cdots
    \end{align*}
    is homotopy equivalent to $\Omega^{\infty}\Sigma^{\infty}\Sigma^{4n - 4}\mathbb{RP}^2$~\cite[Propositions~1.3 and~1.4]{Cohen_1983}. In algebra we have the following commutative diagram of exact sequences
\[\begin{tikzcd}[row sep=small, column sep=small]
	  &                  & 0                & 0                            &   \\
	0 & \Lambda(2n - 1)  & \Lambda(2n)      & \Lambda(4n - 1)              & 0 \\
	0 & \Lambda(2n - 1)  & \Lambda(2n + 1)  & \operatorname{cof}_n(d_1)    & 0 \\
	  &                  & \Lambda(4n + 1)  & \Lambda(4n + 1)              &   \\
	  &                  & 0                & 0                            &
	\arrow[from=2-1, to=2-2] \arrow[from=2-2, to=2-3] \arrow[from=2-3, to=2-4] \arrow[from=2-4, to=2-5]
	\arrow[from=3-1, to=3-2] \arrow[from=3-2, to=3-3] \arrow[from=3-3, to=3-4] \arrow[from=3-4, to=3-5]
	\arrow[from=1-3, to=2-3] \arrow[from=2-3, to=3-3] \arrow[from=3-3, to=4-3] \arrow[from=4-3, to=5-3]
	\arrow[from=1-4, to=2-4] \arrow[from=2-4, to=3-4] \arrow[from=3-4, to=4-4] \arrow[from=4-4, to=5-4]
	\arrow[equals, from=2-2, to=3-2] \arrow[equals, from=4-3, to=4-4]
\end{tikzcd}\]
which shows that the range where $\mathcal{A}_* \underline{\otimes} \operatorname{cof}_n(d_1)$ is acyclic is equal to the acyclicity range of $\mathcal{A}_* \underline{\otimes} \Lambda(4n - 1)$. The main theorem of Harper--Miller~\cite[Theorem~3.9]{Harper_Miller_1989} states that for a simply connected space of finite type with \emph{very nice} cohomology (one whose $\mathbb{F}_2$-cohomology is a free unstable algebra $U(M)$ on an unstable module $M$ over the Steenrod algebra), this acyclicity range is equal to the range in which certain structure maps in the associated unstable Adams resolution are surjective on cohomology. This result was generalized to the above resolution for $\operatorname{fib}_n(\mathrm{E}^2)$ in~\cite[Proposition~2.26]{Thompson_1990} (stated there for odd primes; the proof carries over to $p = 2$). Proposition 6.3 and the lambda-algebra estimates from Proposition 6.6 in~\cite{Mahowald_Thompson_1994} show that any chain map $\mathcal{A}_* \underline{\otimes} \operatorname{cof}_n(d_1) \to \mathcal{A}_* \underline{\otimes} \operatorname{cof}_{n+1}(d_1)$ is uniquely determined up to chain homotopy by its effect in homological degree $0$. In particular, since the map $\operatorname{fib}_n(\mathrm{E}^2) \to \Omega^4\operatorname{fib}_{n+1}(\mathrm{E}^2)$ is an isomorphism on $H^{4n - 3}(-;\mathbb{Z})$, the induced map on resolutions is chain homotopic to the chain map $\mathcal{A}_* \underline{\otimes} \operatorname{cof}_n(d_1) \to \mathcal{A}_* \underline{\otimes} \operatorname{cof}_{n+1}(d_1)$ induced by the maps 
\begin{equation}
\label{eq:secondmap}
\begin{aligned}
  \operatorname{cof}_n(d_1) \to  \operatorname{cof}_{n+1}(d_1) \to  \operatorname{cof}_{n+2}(d_1) \to \cdots
\end{aligned}
\end{equation}
from~\cite[Definition~3.2]{Mahowald_1975}. Composing the map in \eqref{eq:firstmap} with the maps in \eqref{eq:secondmap} gives the result stated in the theorem.

\end{proof}
\begin{data}
  \label{data:C2-comp}
  The plain-text table file \texttt{C2.csv} in the \texttt{data/} directory of~\cite{ehpreprint_data} describes all values of the maps from \Cref{thm:mah} on $E_2$-pages through total degree 76.
\end{data}
\noindent In practice, \Cref{thm:mah} and \Cref{data:C2-comp} give us a map of spectral sequences
\begin{equation}
\label{eq:der}
\begin{aligned}
  E_r^{s, f}(S^{2n + 1}) \xrightarrow{\;\mathrm{s}_{2n+1}\;} E_r^{s - 2n + 1, f - 1}(\mathbb{S}/2)
\end{aligned}
\end{equation}
which we use to prove unstable Adams differentials using the stable Adams
spectral sequence for $\mathbb{S}/2$. For $n = 1$ this map was the central tool in~\cite{Curtis_Mahowald_1989}.
\section{The computation of unstable Adams differentials}
\label{sec:ehp}

\noindent In this section we describe our algorithm for exhaustively applying \Cref{prop:ehp} and \Cref{prop:leibniz} to the propagation of unstable Adams differentials. Given the degree shifts in \Cref{prop:ehp} and \Cref{prop:leibniz}, it helps to consider the unstable Adams spectral sequence for all spheres as a single trigraded spectral sequence where we denote the $\mathbb{F}_2$-vector space $E_r^{s, f}(S^n)$ by $\mathcal{U}_r^{n, s, f}$. Our strategy begins with the observation that stabilization and \Cref{thm:mah} give us maps of spectral sequences
\begin{equation}
  \label{eq:stable}
  \mathcal{U}_r^{n, s, f} \xrightarrow{\;\mathrm{E}_n^\infty\;} E_r^{s, f}(\mathbb{S}) \qquad \qquad \qquad \qquad \mathcal{U}_r^{2n + 1, s, f} \xrightarrow{\;\mathrm{s}_{2n+1}\;} E_r^{s - 2n + 1, f - 1}(\mathbb{S}/2),
\end{equation}
where the targets are stable Adams spectral sequences and $\mathrm{s}_{2n + 1}$ is the map \eqref{eq:der}. The Adams differentials for the sphere spectrum and the cofiber of 2 have been well studied and are completely known at least through the 80-stem~\cite{Isaksen_Wang_Xu_2023, Lin_Wang_Xu_2024, Lin_Wang_Xu_2025}. The fact that these maps commute with differentials provides the initial unstable Adams differentials, which we then propagate using \Cref{prop:ehp} and \Cref{prop:leibniz}.  
\subsection{Affine linear algebra}Our method for propagating unstable Adams differentials using naturality with respect to EHP maps and the unstable Leibniz rule is based on the strategy for stable Adams $d_2$-differentials presented in~\cite{Beauvais_Feisthauer_2022}. Throughout this section we write $\mathbf{d}_r$ for the matrix representing the differential $d_r$ on the tridegree in question. Recall that we can express the space of possible values of $\mathbf{d}_r$ as an affine subspace $\boldsymbol{\delta}_r + \mathcal{I}_r$, where the \emph{offset} $\boldsymbol{\delta}_r \in \operatorname{Hom}(\mathcal{U}_r^{n,s,f}, \mathcal{U}_r^{n,s-1,f+r})$ is a chosen matrix representative and the \emph{indeterminacy} $\mathcal{I}_r \subseteq \operatorname{Hom}(\mathcal{U}_r^{n,s,f}, \mathcal{U}_r^{n,s-1,f+r})$ is a linear subspace.  
\begin{remark*}
  Typically the Leibniz rule is used to take information about a differential $d_r(x)$ and a product $x \circ y$ and deduce something about the differential $d_r(x \circ y)$. Our implementation instead considers the set of all possible matrix representatives for $d_r$ on some tridegree and narrows down our set of possible differentials by excluding matrices for which certain squares fail to commute. The following description in terms of affine linear algebra is meant to make this idea precise.
\end{remark*}
We use the affine linear algebra library in SageMath (version 10.4)~\cite{SageMath} to translate commutative
diagrams involving $\mathbf{d}_r$ into linear constraints on its space of
possible values by computing intersections of affine spaces and taking images
and preimages under fixed linear maps. For each tridegree
$\mathcal{U}_r^{n,s,f}$ we initialize its affine subspace of possible
differentials $\mathcal{D}_r^{n,s,f}$ as $\mathbf{0} +
\operatorname{Hom}_{\mathbb{F}_2}(\mathcal{U}_r^{n,s,f}, \mathcal{U}_r^{n,s-1,f+r})$. For a fixed linear map $\mathrm{F}$ and affine subspace $\mathcal{D}$, define the pushforward and pullback operators
\begin{align*}
  \mathrm{F}_*(\mathcal{D}) &= \{\,\mathrm{F}\circ\mathbf{m} : \mathbf{m} \in \mathcal{D}\,\},
  &
  \mathrm{F}^*(\mathcal{D}) &= \{\,\mathbf{m}\circ\mathrm{F} : \mathbf{m} \in \mathcal{D}\,\}.
\end{align*}

\begin{proposition}
Let $\mathrm{F}$ be a linear map on $\mathcal{U}_r^{\ast,\ast,\ast}$ and denote
its target tridegree on $\mathcal{U}_r^{n,s,f}$ by $\mathrm{F}(n,s,f)$.
Naturality of $\mathrm{F}$ with respect to $\mathbf{d}_r$ is the equation
$\mathrm{F}\circ\mathbf{d}_r = \mathbf{d}_r\circ\mathrm{F}.$ On affine subspaces, naturality with indeterminacy becomes:
\[
\mathbf{d}_r\!\in\!\mathcal{D}_r^{n,s,f}
\;\Rightarrow\;
\mathbf{d}_r\!\in\!(\mathrm{F}^*)^{-1}\!\big(\mathrm{F}_*(\mathcal{D}_r^{n,s,f})\big),
\qquad
\mathbf{d}_r\!\in\!\mathcal{D}_r^{\mathrm{F}(n,s,f)}
\;\Rightarrow\;
\mathbf{d}_r\!\in\!\mathrm{F}_*^{-1}\!\big(\mathrm{F}^*(\mathcal{D}_r^{\mathrm{F}(n,s,f)})\big).
\]
In each implication the hypothesis concerns the differential on the tridegree appearing in its superscript, and the conclusion constrains the differential on the other tridegree.
\end{proposition}
\begin{proof}
  The proof is the same as in~\cite{Beauvais_Feisthauer_2022}.
\end{proof}
Here and throughout, the arrow $\update$ denotes algorithmic refinement: the space of possible differentials named on the left is replaced, in place, by its intersection with the space on the right. We therefore refine $\mathcal{D}_r^{\ast,\ast,\ast}$ by intersecting with the subspace of maps respecting naturality:
\begin{align}
  \label{eq:nat-const}
  \mathcal{D}_r^{n,s,f}
  &\update
  \mathrm{F}_*^{-1}\!\big(\mathrm{F}^*(\mathcal{D}_r^{\mathrm{F}(n,s,f)})\big)
  &&
  \mathcal{D}_r^{\mathrm{F}(n,s,f)}
  \update
  (\mathrm{F}^*)^{-1}\!\big(\mathrm{F}_*(\mathcal{D}_r^{n,s,f})\big).
\end{align}

The goal is to intersect $\mathcal{D}_r^{n,s,f}$ with constraints until there is no indeterminacy:
\[
\mathcal{D}_r^{n,s,f} = \mathbf{d}_r + \{0\}.
\]
\subsection{Naturality of EHP}
In this subsection, we sketch our strategy for exhaustively applying the naturality result in \Cref{prop:ehp}. This result implies that the following diagrams commute:

\begin{center}
\begin{tikzcd}[
    ampersand replacement=\&,
    every label/.append style={font=\small},
    every arrow/.append style={thick},
    row sep=2em, column sep=2.5em,
    cells={nodes={font=\small}}
]
    \mathcal{U}_r^{n,s,f} \&
    \mathcal{U}_r^{n+1,s,f} \\
    \mathcal{U}_r^{n,s-1,f+r} \&
    \mathcal{U}_r^{n+1,s-1,f+r}
    \arrow["\mathrm{E}", from=1-1, to=1-2]
    \arrow["\mathbf{d}_r"', from=1-1, to=2-1]
    \arrow["\mathbf{d}_r", from=1-2, to=2-2]
    \arrow["\mathrm{E}"', from=2-1, to=2-2]
\end{tikzcd}
\quad
\begin{tikzcd}[
    ampersand replacement=\&,
    every label/.append style={font=\small},
    every arrow/.append style={thick},
    row sep=2em, column sep=2.5em,
    cells={nodes={font=\small}}
]
    \mathcal{U}_r^{n,s,f} \&
    \mathcal{U}_r^{2n-1,s-n+1,f-1} \\
    \mathcal{U}_r^{n,s-1,f+r} \&
    \mathcal{U}_r^{2n-1,s-n,f+r-1}
    \arrow["\mathrm{H}", from=1-1, to=1-2]
    \arrow["\mathbf{d}_r"', from=1-1, to=2-1]
    \arrow["\mathbf{d}_r", from=1-2, to=2-2]
    \arrow["\mathrm{H}"', from=2-1, to=2-2]
\end{tikzcd}
\quad
\begin{tikzcd}[
    ampersand replacement=\&,
    every label/.append style={font=\small},
    every arrow/.append style={thick},
    row sep=2em, column sep=2.5em,
    cells={nodes={font=\small}}
]
    \mathcal{U}_r^{2n+1,s,f} \&
    \mathcal{U}_r^{n,s+n-1,f+2} \\
    \mathcal{U}_r^{2n+1,s-1,f+r} \&
    \mathcal{U}_r^{n,s+n-2,f+r+2}
    \arrow["\mathrm{P}", from=1-1, to=1-2]
    \arrow["\mathbf{d}_r"', from=1-1, to=2-1]
    \arrow["\mathbf{d}_r", from=1-2, to=2-2]
    \arrow["\mathrm{P}"', from=2-1, to=2-2]
\end{tikzcd}
\end{center}

\clearpage

%

\begingroup
\newcommand{\legendscale}{1}%
\setlength{\intextsep}{0.1cm}%
\setlength{\textfloatsep}{0.1cm}%
\setlength{\abovecaptionskip}{3pt}%
\setlength{\belowcaptionskip}{0pt}%

\begin{figure}[H]
\centering
\scalebox{\legendscale}{%
\begin{tikzpicture}[>=latex, line join=bevel]
\node (leafpanel) [panel] at (0,0) {
  \begin{tikzpicture}
    \node[inner sep=0] at (-7.75,0) {}; \node[inner sep=0] at (7.75,0) {};
    \node[font=\bfseries] at (0,2) {Terminal Nodes};

    \node (gray) at (-5,1) [ctpnode, fill=CtpOverlay2] {$\mathbf{(n,s,f)}$};
    \node (peach1) at (0,1) [ctpnode, fill=CtpPeach] {$\mathbf{X(s,f)}$};

    \node (red) at (5,1) [ctpnode, fill=CtpRed] {$\mathbf{(n,s,f)}$};
    \node[description] at (-5,-0.2) {$\mathcal{D}_r^{n,s,f} = \textcolor{CtpPeach}{\mathbf{0}} + \textcolor{CtpGreen}{\operatorname{Hom}(\mathcal{U}_r^{n,s,f}, \mathcal{U}_r^{n,s-1,f+r})}$};
    \node[description-small] at (5,-0.2) {Trivial target};
    \node[description-small] at (0,-0.2) {Stable input};
  \end{tikzpicture}
};
\end{tikzpicture}}
\caption{}\label{fig:legend-terminal}
\end{figure}

\vfill

\begin{figure}[H]
\centering
\scalebox{\legendscale}{%
\begin{tikzpicture}[>=latex, line join=bevel]
\node (innerpanel) [panel] at (0,0) {
  \begin{tikzpicture}
    \node[inner sep=0] at (-7.75,0) {}; \node[inner sep=0] at (7.75,0) {};
    \node[font=\bfseries] at (0,2) {Internal Nodes};

    \node (blue) at (-5,1) [ctpnode, fill=CtpSky] {$\mathbf{(n,s,f)}$};
    \node (green) at (0,1) [ctpnode, fill=CtpGreen] {$\mathbf{(n,s,f)}$};
    \node (red2) at (5,1) [ctpnode, fill=CtpRed] {$\mathbf{(n,s,f)}$};

    \node[description-wide] at (-5,-0.2) {$\{0\} < \textcolor{CtpGreen}{\mathcal{I}_r^{n, s, f}} < \operatorname{Hom}\big(\mathcal{U}_r^{n, s, f}, \mathcal{U}_r^{n, s - 1, f + r}\big)$};
    \node[description-small] at (0,-0.2) {$\mathcal{D}_r^{n, s, f} = \textcolor{CtpPeach}{\mathbf{d}_r} + \textcolor{CtpGreen}{\{0\}}$};
    \node[description-small] at (5,-0.2) {$\mathcal{D}_r^{n, s, f} = \textcolor{CtpPeach}{\mathbf{0}} + \textcolor{CtpGreen}{\{0\}}$};
  \end{tikzpicture}
};
\end{tikzpicture}}
\caption{}\label{fig:legend-internal}
\end{figure}

\vfill

\begin{figure}[H]
\centering
\scalebox{\legendscale}{%
\begin{tikzpicture}[>=latex, line join=bevel]
\node (arrowpanel) [panel] at (0,0) {
  \begin{tikzpicture}
    \node[inner sep=0] at (-7.75,0) {}; \node[inner sep=0] at (7.75,0) {};
    \node[font=\bfseries] at (0,2.8) {Maps};

    \begin{scope}[shift={(-4,0)}]
      \node (top1) at (0,2) [ctpnode, fill=CtpBase, text=CtpText] {$\mathbf{(n,s,f)}$};
      \node (left1) at (-1.5,1) [ctpnode, fill=CtpBase, text=CtpText] {$\mathbf{F(n,s,f)}$};
      \node (right1) at (1.5,1) [ctpnode, fill=CtpBase, text=CtpText] {$\mathbf{(n,s,f)}$};

      \draw [CtpText,->, very thick] (top1) to[out=225,in=90]
        node [above left,CtpText,inner sep=2pt] {$\mathbf{F}$} (left1);
      \draw [CtpText, very thick] (top1) to[out=315,in=90] (right1);
    \end{scope}

    \begin{scope}[shift={(4,0)}]
      \node (top2) at (0,2) [ctpnode, fill=CtpBase, text=CtpText] {$\mathbf{F(n,s,f)}$};
      \node (left2) at (-1.5,1) [ctpnode, fill=CtpBase, text=CtpText] {$\mathbf{(n, s, f)}$};
      \node (right2) at (1.5,1) [ctpnode, fill=CtpBase, text=CtpText] {$\mathbf{F(n,s,f)}$};

      \draw [CtpText,->, very thick] (left2) to[out=90,in=225]
        node [above left,CtpText,inner sep=2pt] {$\mathbf{F}$} (top2);
      \draw [CtpText, very thick] (top2) to[out=315,in=90] (right2);
    \end{scope}

    \node[description-leibniz] at (-4,-0.5) {$\mathcal{D}_r^{n, s, f} \update \mathrm{F}_*^{-1}\!\big(\mathrm{F}^*(\mathcal{D}_r^{\mathrm{F}(n, s, f)})\big)$};
    \node[description-leibniz] at (4,-0.5) {$\mathcal{D}_r^{\mathrm{F}(n,s,f)} \update (\mathrm{F}^*)^{-1}\!\big(\mathrm{F}_*(\mathcal{D}_r^{n, s, f})\big)$};
  \end{tikzpicture}
};
\end{tikzpicture}}
\caption{}\label{fig:legend-maps}
\end{figure}

\vfill

\begin{figure}[H]
\centering
\scalebox{\legendscale}{%
\begin{tikzpicture}[>=latex, line join=bevel]
\node (leibnizpanel) [panel] at (0,0) {
  \begin{tikzpicture}
    \node[inner sep=0] at (-7.75,0) {}; \node[inner sep=0] at (7.75,0) {};
    \node[font=\bfseries] at (0,3.0) {Products};

    \begin{scope}[shift={(-4.0,0)}]
      \node (ltop1) at (0,2.2) [ctpnode, fill=CtpBase, text=CtpText] {$\mathbf{(n,s,f)}$};
      \node (lleft1) at (-2.2,1.2) [ctpnode, fill=CtpBase, text=CtpText] {$\mathbf{(n,s,f)}$};
      \node (lmid1) at (0,1.2) [ctpnode, fill=CtpBase, text=CtpText] {$\mathbf{(n{+}s,s',f')}$};
      \node (lright1) at (2.3,1.2) [ctpnode, fill=CtpBase, text=CtpText] {$\mathbf{(n,s{+}s',f{+}f')}$};

      \draw [CtpText, very thick] (ltop1) to[out=270,in=90] (lleft1);
      \draw [CtpText, very thick] (ltop1) to[out=270,in=90] (lmid1);
      \draw [CtpText, very thick] (ltop1) to[out=270,in=90] (lright1);

      \node[description-leibniz] at (0,0.1) {$\mathcal{D}_r^{n,s,f} \update (R_y)_*^{-1}\!\big(R_{\mathrm{E}y}^*(\mathcal{D}_r^{n, s + s', f + f'})\big)$};
    \end{scope}

    \begin{scope}[shift={(4.0,0)}]
      \node (ltop2) at (0,2.2) [ctpnode, fill=CtpBase, text=CtpText] {$\mathbf{(n{+}s,s',f')}$};
      \node (lleft2) at (-2.2,1.2) [ctpnode, fill=CtpBase, text=CtpText] {$\mathbf{(n,s,f)}$};
      \node (lmid2) at (0,1.2) [ctpnode, fill=CtpBase, text=CtpText] {$\mathbf{(n{+}s,s',f')}$};
      \node (lright2) at (2.3,1.2) [ctpnode, fill=CtpBase, text=CtpText] {$\mathbf{(n,s{+}s',f{+}f')}$};

      \draw [CtpText, very thick] (ltop2) to[out=270,in=90] (lleft2);
      \draw [CtpText, very thick] (ltop2) to[out=270,in=90] (lmid2);
      \draw [CtpText, very thick] (ltop2) to[out=270,in=90] (lright2);

      \node[description-leibniz] at (0,0.1) {$\mathcal{D}_r^{n+s,s',f'} \update \bigl((L_x)_* \circ \mathrm{E}^*\bigr)^{-1}\!\big((L_x \circ \mathrm{E})^*(\mathcal{D}_r^{n, s + s', f + f'})\big)$};
    \end{scope}

    \begin{scope}[shift={(0,-3.3)}]
      \node (ltop3) at (0,2.2) [ctpnode, fill=CtpBase, text=CtpText] {$\mathbf{(n,s{+}s',f{+}f')}$};
      \node (lleft3) at (-2.2,1.2) [ctpnode, fill=CtpBase, text=CtpText] {$\mathbf{(n,s,f)}$};
      \node (lmid3) at (0,1.2) [ctpnode, fill=CtpBase, text=CtpText] {$\mathbf{(n{+}s,s',f')}$};
      \node (lright3) at (2.3,1.2) [ctpnode, fill=CtpBase, text=CtpText] {$\mathbf{(n,s{+}s',f{+}f')}$};

      \draw [CtpText, very thick] (ltop3) to[out=270,in=90] (lleft3);
      \draw [CtpText, very thick] (ltop3) to[out=270,in=90] (lmid3);
      \draw [CtpText, very thick] (ltop3) to[out=270,in=90] (lright3);

      \node[description, text width=12.2cm] at (0,0.1) {$\mathcal{D}_r^{n, s+s', f+f'} \update \bigl((M \circ (\mathrm{id} \otimes \mathrm{E}))^*\bigr)^{-1}\big(M_*(\mathcal{D}_r^{n, s, f} \otimes \mathrm{id}) + M_*(\mathrm{id} \otimes \mathrm{E}^*(\mathcal{D}_r^{n + s, s', f'}))\big)$};
    \end{scope}
  \end{tikzpicture}
};
\end{tikzpicture}}
\caption{}\label{fig:legend-products}
\end{figure}

\endgroup

\clearpage

\noindent We can constrain the possible values of a differential on some unstable tridegree by iteratively applying \eqref{eq:nat-const} where $\mathrm{F}$ is the matrix form of the maps $\mathrm{E}$, $\mathrm{H}$, and $\mathrm{P}$ computed in \Cref{data:ehp-comp}. With no additional outside information, this procedure can only compute differentials that are zero. To compute nontrivial unstable Adams differentials,
we assume knowledge of all differentials in the Adams spectral sequences for $\mathbb{S}$ and $\mathbb{S}/2$ through the 80-stem~\cite{Isaksen_Wang_Xu_2023, Lin_Wang_Xu_2024, Lin_Wang_Xu_2025} and lift their differentials to the unstable Adams spectral sequence via the maps in \eqref{eq:stable}, the second of which is computed in \Cref{data:C2-comp}. We present these deductions as flow charts, whose legend is given in \Crefrange{fig:legend-terminal}{fig:legend-products}. The flow chart for any computed differential can be regenerated from the released data with the \texttt{why.py} script in~\cite{ehpreprint_data}.

\begin{algorithm}[htbp]
\caption{How a differential's proof is recorded as a flow chart.
\textsc{Propagate} shrinks the space $\mathcal{D}_r^{n,s,f}$ of still-possible
values of $\mathbf{d}_r$ at every tridegree by intersecting it with the
naturality and Leibniz constraints, \emph{recording} each refinement together
with the neighboring tridegrees it used. \textsc{Why} then traces these records
backward from the target differential to the assumed inputs.}
\label{alg:flowchart}
\begin{algorithmic}[1]
\Procedure{Propagate}{$r$}
  \State $\mathcal{D}_r^{n,s,f}\gets\operatorname{Hom}\!\big(\mathcal{U}_r^{n,s,f},\mathcal{U}_r^{n,s-1,f+r}\big)$ for all $(n,s,f)$ \Comment{every value still possible}
  \State restrict $\mathcal{D}_r^{n,s,f}$ to the known differential at each assumed input \Comment{stable $\mathbb{S}$ and $\mathbb{S}/2$}
  \Repeat
    \ForAll{maps $\mathrm{F}\in\{\mathrm{E},\mathrm{H},\mathrm{P}\}$ and tridegrees $(n,s,f)$}\Comment{naturality, \Cref{prop:ehp}}
      \State $\mathcal{D}_r^{n,s,f}\update\{\,\text{maps making the $\mathrm{F}$-naturality square commute}\,\}$
      \State \textbf{if} this shrank $\mathcal{D}_r^{n,s,f}$\textbf{,} record the step $\big(\mathrm{F};\ \mathrm{F}(n,s,f)\big)$ at $(n,s,f)$
    \EndFor
    \ForAll{products $x\!\cdot\! y$ landing in $(n,s{+}s',f{+}f')$}\Comment{Leibniz, \Cref{prop:leibniz}}
      \State refine $\mathcal{D}_r^{n,s{+}s',f{+}f'}$ to force $\mathbf{d}_r(xy)=\mathbf{d}_r(x)\,y+x\,\mathbf{d}_r(y)$
      \State \textbf{if} this shrank it\textbf{,} record the step $\big(x\!\cdot\! y;\ x,\ y\big)$ there
    \EndFor
  \Until{no $\mathcal{D}_r^{n,s,f}$ changes}
\EndProcedure
\Statex
\Function{Why}{$w$}\Comment{$w=(n,s,f)$; call on the target to be proved}
  \State $\sigma\gets$ the most recent step recorded at $w$ \Comment{empty if $w$ was never refined}
  \State create node $N_w$, colored by the state of $\mathcal{D}_r^{w}$ at $\sigma$ \Comment{\Cref{fig:legend-internal}}
  \If{$\sigma$ is empty \textbf{ or } $w$ is an assumed input}
    \State \Return $N_w$ \Comment{terminal: a stable/$\mathbb{S}/2$ input, or an already-forced group}
  \ElsIf{$\sigma$ is a naturality step for a map $\mathrm{F}$}
    \State attach \Call{Why}{$\mathrm{F}(w)$} to $N_w$ on an edge labeled $\mathrm{F}$
  \Else{ \, $\sigma$ is the Leibniz step for a product $x\!\cdot\! y$}
    \State attach \Call{Why}{$x$} and \Call{Why}{$y$} to $N_w$
  \EndIf
  \State \Return $N_w$
\EndFunction
\end{algorithmic}
\end{algorithm}

\noindent\emph{Reading a flow chart.} The root is the differential being proved
and the terminal nodes are the inputs we assume: a stable sphere differential
$\mathbb{S}(s,f)$, a differential on the $\mathbb{S}/2$ column $\mathrm{C2}(s,f)$, or a group on
which $\mathbf{d}_r$ is already forced. Every edge is a single instance of
naturality (\Cref{prop:ehp}) or the Leibniz rule (\Cref{prop:leibniz}), so
following the labeled maps $\mathrm{E},\mathrm{H},\mathrm{P}$ from the terminal
nodes up to the root replays exactly the chain of constraints that pins down the
value at the root. A node's color records how determined its differential was
when it was used: gray unconstrained, blue partially constrained, green once a
unique nonzero value is forced, red once it is forced to vanish. Each
refined node also points back to its own earlier, less-determined state (the
adjacent gray or blue node), so the chart shows $\mathcal{D}_r$ shrinking in place.
Because every step is valid and the terminal nodes are assumptions, the chart is
a complete proof of the differential at its root, checkable edge by edge.

The terminal nodes (\Cref{fig:legend-terminal}) mark the places where a proof begins. An
orange node is the differential in some bidegree of a stable Adams spectral
sequence, for either the sphere spectrum $\mathbb{S}$ or the mod-$2$ Moore
spectrum $\mathbb{S}/2$; these are the inputs we assume. A red terminal node is a
differential that vanishes because its target is the trivial vector space, and a
gray terminal node is a completely unknown differential, one whose space of
possible values is still the full space of linear maps from the source to the
target. The internal nodes (\Cref{fig:legend-internal}) record differentials that the propagation has
already constrained. A green node is a differential that is fully known and
nonzero, and a red node is one that is fully known and zero; here \emph{fully
known} means that the space of possible matrices has been narrowed to a single
element. A blue node is a differential that is only partially known, in the
sense that some matrices have been ruled out but more than one possibility
remains. The diagrams in \Cref{fig:legend-maps,fig:legend-products} record the two kinds of constraint. In the
left-hand diagram of \Cref{fig:legend-maps}, a map
$\mathrm{F}\colon \mathcal{U}_r^{n,s,f} \to \mathcal{U}_r^{\mathrm{F}(n,s,f)}$ lets us use what we
know about the differential on $\mathcal{U}_r^{\mathrm{F}(n,s,f)}$ to refine the
differential on $\mathcal{U}_r^{n,s,f}$; the right-hand diagram is the mirror
image, refining $\mathcal{U}_r^{\mathrm{F}(n,s,f)}$ using $\mathcal{U}_r^{n,s,f}$. \Cref{fig:legend-products}
shows the analogous refinements coming from the Leibniz rule for a product
$x\cdot y$.
\begin{example}
\label{ex:simple}
Among unstable Adams differentials that stabilize to zero, the one of lowest
stem occurs on $\mathcal{U}_2^{5,20,5}$. The commutativity of the following diagram shows how the differential on $\mathcal{U}_2^{5, 20, 5}$ can be obtained by pulling back the stable Adams differential on $h_0 f_0$ through iterative applications of the maps $\mathrm{H}$ and $\mathrm{P}$. Note that all unstable tridegrees in the diagram have a single generator and each horizontal map is nontrivial on that generator.

\begin{center}
\begin{tikzcd}[
    ampersand replacement=\&,
    every label/.append style={font=\small},
    every arrow/.append style={thick},
    row sep=2em, column sep=2.5em,
    cells={nodes={font=\small}}
]
    {\mathcal{U}_2^{5,20,5}} \& 
    {\mathcal{U}_2^{2,21,7}} \& 
    {\mathcal{U}_2^{3,20,6}} \& 
    {\mathcal{U}_2^{5,18,5}} \& 
    {E_2^{18,5}(\mathbb{S}) \cong \mathbb{F}_2\{h_0 f_0\}} \\
    {\mathcal{U}_2^{5,19,7}} \& 
    {\mathcal{U}_2^{2,20,9}} \& 
    {\mathcal{U}_2^{3,19,8}} \& 
    {\mathcal{U}_2^{5,17,7}} \& 
    {E_2^{17,7}(\mathbb{S}) \cong \mathbb{F}_2\{h_0^3 e_0\}}
    \arrow["{\mathrm{P}}", pos=0.4, from=1-1, to=1-2]
    \arrow["{\mathbf{d}_2}"', pos=0.4, dashed, from=1-1, to=2-1]
    \arrow["{\mathrm{H}}", pos=0.4, from=1-2, to=1-3]
    \arrow["{\mathbf{d}_2}", pos=0.4, dashed, from=1-2, to=2-2]
    \arrow["{\mathrm{H}}", pos=0.4, from=1-3, to=1-4]
    \arrow["{\mathbf{d}_2}", pos=0.4, dashed, from=1-3, to=2-3]
    \arrow["{\mathrm{E}_5^\infty}", pos=0.3, above, from=1-4, to=1-5]
    \arrow["{\mathbf{d}_2}", pos=0.4, dashed, from=1-4, to=2-4]
    \arrow["{\mathbf{d}_2} = {[1]}", pos=0.4, from=1-5, to=2-5]
    \arrow["{\mathrm{P}}"', pos=0.4, from=2-1, to=2-2]
    \arrow["{\mathrm{H}}"', pos=0.4, from=2-2, to=2-3]
    \arrow["{\mathrm{H}}"', pos=0.4, from=2-3, to=2-4]
    \arrow["{\mathrm{E}_5^\infty}"', pos=0.3, below, from=2-4, to=2-5]
\end{tikzcd}
\end{center}

\noindent The stable group $E_2^{18,5}(\mathbb{S})$ is generated by $h_0 f_0$, which supports the known stable differential $d_2(h_0 f_0) = h_0^3 e_0$. Since the stabilization map $\mathrm{E}_5^\infty$ carries the generator of $\mathcal{U}_2^{5,18,5}$ to $h_0 f_0$ and the generator of $\mathcal{U}_2^{5,17,7}$ to $h_0^3 e_0$, commutativity of the rightmost square forces the differential on $\mathcal{U}_2^{5,18,5}$ to be nonzero. Repeating this argument square by square from right to left (each of the maps $\mathrm{H}$, $\mathrm{H}$, and $\mathrm{P}$ commutes with $\mathbf{d}_2$ and is nonzero on the relevant generator) forces the differentials on $\mathcal{U}_2^{3,20,6}$, on $\mathcal{U}_2^{2,21,7}$, and finally on $\mathcal{U}_2^{5,20,5}$ to be nonzero as well. Using the notational schema presented above, we can abbreviate this commutative diagram as a flow chart given in \Cref{fig:flowchart1}.

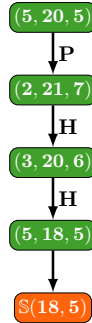
\begin{figure}[htbp]
\centering
   \scalebox{\flowchartscale}{%
   \begin{tikzpicture}[>=latex,line join=bevel]
      \node (n_5_20_5) at (0,0) [ctpnode,fill=CtpGreen] {$\mathbf{(5, 20, 5)}$};
      \node (n_2_21_7) at (0,-1.2) [ctpnode,fill=CtpGreen] {$\mathbf{(2, 21, 7)}$};
      \node (n_3_20_6) at (0,-2.4) [ctpnode,fill=CtpGreen] {$\mathbf{(3, 20, 6)}$};
      \node (n_5_18_5) at (0,-3.6) [ctpnode,fill=CtpGreen] {$\mathbf{(5, 18, 5)}$};
      \node (n_Stable_5_18_5) at (0,-4.8) [ctpnode,fill=CtpPeach] {$\mathbf{\mathbb{S}(18, 5)}$};

      \draw [CtpText, ->, very thick] (n_5_20_5) to[out=270,in=90] node [right,CtpText,inner sep=2pt]
    {$\mathbf{P}$} (n_2_21_7);
      \draw [CtpText, ->, very thick] (n_2_21_7) to[out=270,in=90] node [right,CtpText,inner sep=2pt]
    {$\mathbf{H}$} (n_3_20_6);
      \draw [CtpText, ->, very thick] (n_3_20_6) to[out=270,in=90] node [right,CtpText,inner sep=2pt]
    {$\mathbf{H}$} (n_5_18_5);
      \draw [CtpText, ->, very thick] (n_5_18_5) to[out=270,in=90] (n_Stable_5_18_5);
    \end{tikzpicture}}
\caption{Flow chart for $\mathcal{D}_2^{5, 20, 5}$.}
\label{fig:flowchart1}
\end{figure}
\end{example}

The proof in \Cref{ex:simple} is a relatively simple example of a multistep proof using EHP naturality. For more complicated arguments, our abbreviated notation allows us to efficiently communicate the inputs, logical steps, and output of a proof in a way that remains readable even when the corresponding commutative diagram becomes too large and convoluted to follow. 
\begin{example}
\label{ex:4_20_6}
There is a nontrivial differential on $\mathcal{U}_2^{4, 20, 6}$ that can be computed using stable differentials and applications of the maps $\mathrm{E}$ and $\mathrm{P}$. The flow chart for this computation is given in \Cref{fig:flowchart2}. The differential on $\mathcal{U}_2^{9, 17, 4}$ is completely determined by stable information using the maps $\mathrm{E}_9^\infty$ and $\mathrm{s}_9$. Similarly, there is a nontrivial differential on $\mathcal{U}_2^{3, 20, 6}$ that is lifted from the nontrivial differential on $E_2^{19, 5}(\mathbb{S}/2)$. These nontrivial differentials can be pushed to $\mathcal{U}_2^{4, 20, 6}$ through the maps $\mathrm{P}$ and $\mathrm{E}$, and the remaining possibilities for the differential on $\mathcal{U}_2^{4, 20, 6}$ can be ruled out using naturality of the stabilization map. 

\begin{figure}[htbp]
\centering
   \scalebox{\flowchartscale}{%
   \begin{tikzpicture}[>=latex,line join=bevel]
      \node (n_4_20_6_66318) at (0,0) [ctpnode,fill=CtpGreen] {$\mathbf{(4, 20, 6)}$};
      \node (n_9_17_4_17990) at (-3,-1.2) [ctpnode,fill=CtpSky] {$\mathbf{(9, 17, 4)}$};
      \node (n_C2_9_17_4_17990) at (-4,-2.4) [ctpnode,fill=CtpPeach] {$\mathbf{C2(10, 3)}$};
      \node (n_Stable_9_17_4_17990) at (-2,-2.4) [ctpnode,fill=CtpPeach] {$\mathbf{\mathbb{S}(17, 4)}$};
      \node (n_4_20_6_17990) at (3,-1.2) [ctpnode,fill=CtpSky] {$\mathbf{(4, 20, 6)}$};
      \node (n_3_20_6_16743) at (1,-2.4) [ctpnode,fill=CtpSky] {$\mathbf{(3, 20, 6)}$};
      \node (n_C2_3_20_6_16743) at (1,-3.5999999999999996) [ctpnode,fill=CtpPeach] {$\mathbf{C2(19, 5)}$};
      \node (n_4_20_6_16743) at (5,-2.4) [ctpnode,fill=CtpSky] {$\mathbf{(4, 20, 6)}$};
      \node (n_Stable_4_20_6_16743) at (5,-3.5999999999999996) [ctpnode,fill=CtpPeach] {$\mathbf{\mathbb{S}(20, 6)}$};

      \draw [CtpText, <-, very thick] (n_4_20_6_66318) to[out=260,in=90] node [above left,CtpText,inner sep=2pt]
    {$\mathbf{P}$} (n_9_17_4_17990);
      \draw [CtpText, ->, very thick] (n_9_17_4_17990) to[out=250,in=90] (n_C2_9_17_4_17990);
      \draw [CtpText, ->, very thick] (n_9_17_4_17990) to[out=290,in=90] (n_Stable_9_17_4_17990);
      \draw [CtpText, very thick] (n_4_20_6_66318) to[out=270,in=90] (n_4_20_6_17990);
      \draw [CtpText, <-, very thick] (n_4_20_6_17990) to[out=260,in=90] node [above left,CtpText,inner sep=2pt]
    {$\mathbf{E}$} (n_3_20_6_16743);
      \draw [CtpText, ->, very thick] (n_3_20_6_16743) to[out=270,in=90] (n_C2_3_20_6_16743);
      \draw [CtpText, very thick] (n_4_20_6_17990) to[out=270,in=90] (n_4_20_6_16743);
      \draw [CtpText, ->, very thick] (n_4_20_6_16743) to[out=270,in=90] (n_Stable_4_20_6_16743);
    \end{tikzpicture}}
\caption{Flow chart for $\mathcal{D}_2^{4, 20, 6}$.}
\label{fig:flowchart2}
\end{figure}
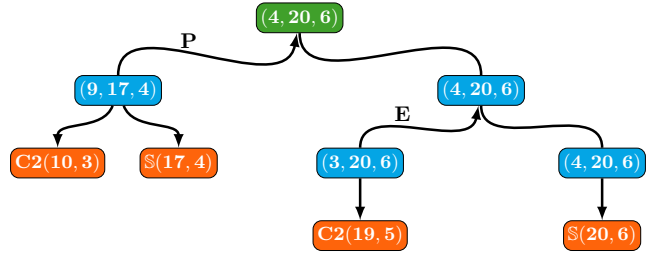
\end{example}

\subsection{The unstable Leibniz rule}

\noindent The unstable Leibniz rule from \Cref{prop:leibniz} states that the following diagram commutes, where $M$ is the algebraic composition map from \Cref{prop:leibniz}
\begin{equation*}
\begin{tikzcd}[ampersand replacement=\&,
    every label/.append style={font=\small},
    every arrow/.append style={thick},
    row sep=2.5em, column sep=4em]
{\mathcal{U}_r^{n, s, f} \otimes \mathcal{U}_r^{n + s - 1, s', f'}} \& {\mathcal{U}_r^{n, s+s', f+f'}} \\
{(\mathcal{U}_r^{n, s-1, f+r} \otimes \mathcal{U}_r^{n + s - 1, s', f'}) \oplus (\mathcal{U}_r^{n, s, f} \otimes \mathcal{U}_r^{n+s, s'-1, f'+r})} \& {\mathcal{U}_r^{n, s+s'-1, f+f'+r}}
\arrow["{M \circ (\mathrm{id} \otimes \mathrm{E})}", from=1-1, to=1-2]
\arrow["{\begin{bmatrix}
    \mathbf{d}_r \otimes \mathrm{id}  \\ \mathrm{id} \otimes (\mathbf{d}_r \circ \mathrm{E})
\end{bmatrix}\quad}"',
    from=1-1, to=2-1]
\arrow["{\mathbf{d}_r}", from=1-2, to=2-2]
\arrow["{\begin{bmatrix} M & M \end{bmatrix}}"',
    from=2-1, to=2-2]
\end{tikzcd}
\end{equation*}
The commutativity of the full Leibniz diagram constrains the possible values of $\mathbf{d}_r$ on $\mathcal{U}_r^{n, s+s', f+f'}$ by updating $\mathcal{D}_r^{n, s+s', f+f'}$ according to
\begin{align*}
\mathcal{D}_r^{n, s+s', f+f'} &\update \bigl((M \circ (\mathrm{id} \otimes \mathrm{E}))^*\bigr)^{-1}\big(M_*(\mathcal{D}_r^{n, s, f} \otimes \mathrm{id}) + M_*(\mathrm{id} \otimes \mathrm{E}^*(\mathcal{D}_r^{n + s, s', f'}))\big).
\end{align*}
These constraints allow us to take information about Adams differentials in earlier stems and propagate it into later stems. We can similarly use the Leibniz rule to take information about Adams differentials in later stems and back-propagate it into earlier stems.
If $x$ is a cycle on the unstable Adams $E_r$-page, the unstable Leibniz rule reduces to the first of the following equations; if $\mathrm{E}y$ is a cycle, it reduces to the second:
\begin{align*}
  x \circ \mathbf{d}_r(\mathrm{E}y) = \mathbf{d}_r(x \circ \mathrm{E}y)
\qquad\qquad\qquad
\mathbf{d}_r(x) \circ y = \mathbf{d}_r(x \circ \mathrm{E}y).
\end{align*}
\noindent We refer to these two specializations of the unstable Leibniz rule as the \emph{partial Leibniz rule}.

\noindent Because algebraic compositions in the unstable Adams spectral sequence are not commutative, we need to keep track of whether we are multiplying on the left or on the right. Given $x \in \mathcal{U}_r^{n, s, f}$ and $y \in \mathcal{U}_r^{n + s, s', f'}$, we define the left and right multiplication maps
\begin{equation*}
L_x : \mathcal{U}_r^{n + s, s', f'} \to \mathcal{U}_r^{n, s + s', f + f'} \qquad R_y : \mathcal{U}_r^{n, s, f} \to \mathcal{U}_r^{n, s + s', f + f'}.
\end{equation*}

\noindent The partial Leibniz rule then implies that the following diagrams commute for
\(x \in \mathcal{U}_r^{n, s, f}\) and
\(y \in \mathcal{U}_r^{n+s-1, s', f'}\). Note that $y$ now lives one sphere lower than in the composition pairing, because the partial Leibniz rule gives information about the differential on $\mathrm{E}y \in \mathcal{U}_r^{n+s, s', f'}$ rather than on $y$ itself. The left diagram commutes when $x$ is a cycle, and the right diagram commutes when $\mathrm{E}y$ is a cycle.

\begin{center}
\begin{tikzcd}[ampersand replacement=\&,
    every label/.append style={font=\small},
    every arrow/.append style={thick},
    row sep=2.5em, column sep=4em]
    {\mathcal{U}_r^{n + s - 1, s', f'}}
        \& {\mathcal{U}_r^{n, s+s', f+f'}}
        \& {\mathcal{U}_r^{n, s, f}}
        \& {\mathcal{U}_r^{n, s+s', f+f'}} \\
    {\mathcal{U}_r^{n+s, s'-1, f'+r}}
        \& {\mathcal{U}_r^{n, s+s'-1, f+f'+r}}
        \& {\mathcal{U}_r^{n, s-1, f+r}}
        \& {\mathcal{U}_r^{n, s+s'-1, f+f'+r}}

\arrow["{L_x \circ \mathrm{E}}", from=1-1, to=1-2]
\arrow[" \mathbf{d}_r \circ \mathrm{E}"',
    from=1-1, to=2-1]
\arrow["{\mathbf{d}_r}", from=1-2, to=2-2]
\arrow["{L_x}"',
    from=2-1, to=2-2]

\arrow["{R_{\mathrm{E}y}}", from=1-3, to=1-4]
\arrow["{\mathbf{d}_r}"',
    from=1-3, to=2-3]
\arrow["{\mathbf{d}_r}", from=1-4, to=2-4]
\arrow["{R_y}"',
    from=2-3, to=2-4]
\end{tikzcd}
\end{center}

\noindent On the level of affine subspaces we can express the partial Leibniz rule as a constraint on the space of possible differentials in the following way:
\begin{equation*}
  \mathcal{D}_r^{n + s,s',f'} \update \bigl((L_x)_* \circ \mathrm{E}^*\bigr)^{-1}\!\bigl((L_x \circ \mathrm{E})^*(\mathcal{D}_r^{n, s + s', f + f'})\bigr)
  \qquad
  \mathcal{D}_r^{n,s,f} \update (R_y)_*^{-1}\!\bigl(R_{\mathrm{E}y}^*(\mathcal{D}_r^{n, s + s', f + f'})\bigr)
\end{equation*}
These updates back-propagate information from the product tridegree into its factors; forward propagation into the product tridegree is already handled by the full Leibniz update above.
The algebraic compositions computed in \Cref{data:ehp-comp} provide a full description of the maps $M, \mathrm{E}$, $L_x$, $R_y$, and $R_{\mathrm{E}y}$ through total degree 76.

\begin{remark*}
  It is possible to write down commutative diagrams based on the full Leibniz rule that isolate the differentials on $x$ or $\mathrm{E}y$ without assuming that either factor tridegree consists of cycles; the analogous diagrams for the stable Leibniz rule are found in \cite{Beauvais_Feisthauer_2022}. Our implementation applies these constraints when neither partial rule is available, but the linear algebra involved in writing them down explicitly is more complicated, and we omit the formulas for brevity.
\end{remark*}
\begin{example}
  \label{ex:prod}

  There is a nontrivial differential on $\mathcal{U}_3^{9, 45, 11}$ that can be
  computed using the algebraic unstable compositions
  \begin{equation*}
    [\lambda_8\lambda_4\lambda_1^3] \circ [\lambda_{12}\lambda_4\lambda_5\lambda_3^3]
      = [\lambda_8\lambda_2^2\lambda_4\lambda_3^3\lambda_6^2\lambda_5\lambda_3]
    \qquad\qquad
    [\lambda_8\lambda_4\lambda_1^3] \circ [\lambda_2^4\lambda_3\lambda_5\lambda_7\lambda_3^2]
      = [\lambda_4\lambda_2^7\lambda_3^2\lambda_6^2\lambda_5\lambda_3]
  \end{equation*}
  together with the stable input differential $d_3(\Delta h_2^2) = h_1 d_0^2$, where $\Delta h_2^2 = [\lambda_{12}\lambda_4\lambda_5\lambda_3^3]$ and $h_1 d_0^2 = [\lambda_7\lambda_6\lambda_2\lambda_3\lambda_4^2\lambda_1^3]$.
\begin{figure}[htbp]
\centering
\scalebox{\flowchartscale}{%
\begin{tikzpicture}[>=latex,line join=bevel]
  \node (n_9_45_11_20632) at (0,0) [ctpnode,fill=CtpGreen] {$\mathbf{(9, 45, 11)}$};
  \node (n_9_15_5_18201) at (-2,-1.2) [ctpnode,fill=CtpRed] {$\mathbf{(9, 15, 5)}$};
  \node (n_24_30_6_18201) at (0.0,-1.2) [ctpnode,fill=CtpSky] {$\mathbf{(24, 30, 6)}$};
  \node (n_Stable_24_30_6_18201) at (0,-2.4) [ctpnode,fill=CtpPeach] {$\mathbf{\mathbb{S}(30, 6)}$};
  \node (n_9_45_11_18201) at (2,-1.2) [ctpnode,fill=CtpOverlay2] {$\mathbf{(9, 45, 11)}$};

  \draw [CtpText, very thick] (n_9_45_11_20632) to[out=270,in=90] (n_9_15_5_18201);
  \draw [CtpText, very thick] (n_9_45_11_20632) to[out=270,in=90] (n_24_30_6_18201);
  \draw [CtpText, ->, very thick] (n_24_30_6_18201) to[out=270,in=90] (n_Stable_24_30_6_18201);
  \draw [CtpText, very thick] (n_9_45_11_20632) to[out=270,in=90] (n_9_45_11_18201);
\end{tikzpicture}}
\caption{Flow chart for $\mathcal{D}_3^{9, 45, 11}$.}
\end{figure}
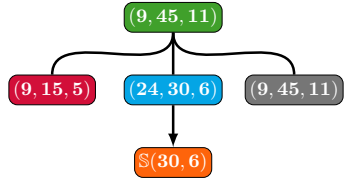

\end{example}

\noindent By combining deductions using EHP naturality with deductions using the unstable Leibniz rule, we can construct increasingly complex arguments for challenging differentials.
\begin{example}
  The following flow chart illustrates our process of applying constraints to $\mathcal{D}_3^{9, 37, 7}$. This computation uses constraints coming from both EHP naturality and the unstable Leibniz rule. One of the classes in $\mathcal{U}_{3}^{9, 37, 7}$ is a cycle because it is the image under $\mathrm{H}$ of a cycle in $\mathcal{U}_3^{5, 41, 8}$. The other generator of $\mathcal{U}_3^{9, 37, 7}$ supports a nontrivial $d_3$-differential. The input data for this computation consists of nontrivial $d_3$-differentials on $E_3^{44, 10}(\mathbb{S})$ and $E_3^{52, 12}(\mathbb{S}/2)$. The maps $\mathrm{E}_{17}^{\infty}$ and $\mathrm{s}_3$ pull these differentials back to $\mathcal{U}_3^{17, 44, 10}$ and $\mathcal{U}_3^{3, 53, 13}$ respectively. By pushing and pulling these differentials through the EHP maps, we can deduce the nontrivial differential on $\mathcal{U}_3^{9, 51, 11}$. The product
  \begin{equation*}
    [\lambda_7\lambda_{12}\lambda_4\lambda_5\lambda_3^3] \circ [\lambda_6\lambda_2\lambda_3^2]
      = [\lambda_7\lambda_2\lambda_4\lambda_5^3\lambda_3\lambda_6^2\lambda_5\lambda_3]
  \end{equation*}
  together with the fact that $[\lambda_6\lambda_2\lambda_3^2] \in \mathcal{U}_3^{46, 14, 4}$ is a cycle gives the
  differential on $[\lambda_7\lambda_{12}\lambda_4\lambda_5\lambda_3^3] \in \mathcal{U}_3^{9, 37, 7}$.

\begin{figure}[!htbp]
  \centering
  {\tikzset{every picture/.append style={execute at end picture={%
    \path let \p1=(current bounding box.west), \p2=(current bounding box.east)
      in (-\x1,\y1) (-\x2,\y2);}}}%
  \scalebox{\flowchartscale}{\begin{tikzpicture}[>=latex,line join=bevel]
      \node (n_9_37_7_20632) at (0,0) [ctpnode,fill=CtpGreen] {$\mathbf{(9, 37, 7)}$};
      \node (n_9_37_7_18206) at (-3.0,-1.2) [ctpnode,fill=CtpSky] {$\mathbf{(9, 37, 7)}$};
      \node (n_5_41_8_5854) at (-3.0,-2.4) [ctpnode,fill=CtpRed] {$\mathbf{(5, 41, 8)}$};
      \node (n_46_14_4_18206) at (0.0,-1.2) [ctpnode,fill=CtpRed] {$\mathbf{(46, 14, 4)}$};
      \node (n_9_51_11_18206) at (3.0,-1.2) [ctpnode,fill=CtpGreen] {$\mathbf{(9, 51, 11)}$};
      \node (n_4_54_13_17000) at (3.0,-2.4) [ctpnode,fill=CtpGreen] {$\mathbf{(4, 54, 13)}$};
      \node (n_7_51_12_16928) at (3.0,-3.5999999999999996) [ctpnode,fill=CtpGreen] {$\mathbf{(7, 51, 12)}$};
      \node (n_8_51_12_16898) at (3.0,-4.8) [ctpnode,fill=CtpGreen] {$\mathbf{(8, 51, 12)}$};
      \node (n_17_44_10_6868) at (1.5,-6.0) [ctpnode,fill=CtpSky] {$\mathbf{(17, 44, 10)}$};
      \node (n_Stable_17_44_10_6868) at (1.5,-7.2) [ctpnode,fill=CtpPeach] {$\mathbf{\mathbb{S}(44, 10)}$};
      \node (n_8_51_12_6868) at (4.5,-6.0) [ctpnode,fill=CtpSky] {$\mathbf{(8, 51, 12)}$};
      \node (n_7_51_12_6045) at (4.5,-7.2) [ctpnode,fill=CtpSky] {$\mathbf{(7, 51, 12)}$};
      \node (n_6_51_12_5946) at (4.5,-8.4) [ctpnode,fill=CtpGreen] {$\mathbf{(6, 51, 12)}$};
      \node (n_5_51_12_5883) at (4.5,-9.6) [ctpnode,fill=CtpGreen] {$\mathbf{(5, 51, 12)}$};
      \node (n_3_53_13_5736) at (4.5,-10.799999999999999) [ctpnode,fill=CtpSky] {$\mathbf{(3, 53, 13)}$};
      \node (n_C2_3_53_13_5736) at (4.5,-11.999999999999998) [ctpnode,fill=CtpPeach] {$\mathbf{C2(52, 12)}$};

      \draw [CtpText, very thick] (n_9_37_7_20632) to[out=260,in=90] (n_9_37_7_18206);
      \draw [CtpText, <-, very thick] (n_9_37_7_18206) to[out=270,in=90] node [midway,left,CtpText,inner sep=2pt]
    {$\mathbf{H}$} (n_5_41_8_5854);
      \draw [CtpText, very thick] (n_9_37_7_20632) to[out=270,in=90] (n_46_14_4_18206);
      \draw [CtpText, very thick] (n_9_37_7_20632) to[out=270,in=90] (n_9_51_11_18206);
      \draw [CtpText, ->, very thick] (n_9_51_11_18206) to[out=270,in=90] node [midway,left,CtpText,inner sep=2pt]
    {$\mathbf{P}$} (n_4_54_13_17000);
      \draw [CtpText, ->, very thick] (n_4_54_13_17000) to[out=270,in=90] node [midway,left,CtpText,inner sep=2pt]
    {$\mathbf{H}$} (n_7_51_12_16928);
      \draw [CtpText, ->, very thick] (n_7_51_12_16928) to[out=270,in=90] node [midway,left,CtpText,inner sep=2pt]
    {$\mathbf{E}$} (n_8_51_12_16898);
      \draw [CtpText, <-, very thick] (n_8_51_12_16898) to[out=260,in=90] node [above left,CtpText,inner sep=2pt]
    {$\mathbf{P}$} (n_17_44_10_6868);
      \draw [CtpText, ->, very thick] (n_17_44_10_6868) to[out=270,in=90] (n_Stable_17_44_10_6868);
      \draw [CtpText, very thick] (n_8_51_12_16898) to[out=270,in=90] (n_8_51_12_6868);
      \draw [CtpText, <-, very thick] (n_8_51_12_6868) to[out=270,in=90] node [midway,left,CtpText,inner sep=2pt]
    {$\mathbf{E}$} (n_7_51_12_6045);
      \draw [CtpText, <-, very thick] (n_7_51_12_6045) to[out=270,in=90] node [midway,left,CtpText,inner sep=2pt]
    {$\mathbf{E}$} (n_6_51_12_5946);
      \draw [CtpText, <-, very thick] (n_6_51_12_5946) to[out=270,in=90] node [midway,left,CtpText,inner sep=2pt]
    {$\mathbf{E}$} (n_5_51_12_5883);
      \draw [CtpText, <-, very thick] (n_5_51_12_5883) to[out=270,in=90] node [midway,left,CtpText,inner sep=2pt]
    {$\mathbf{H}$} (n_3_53_13_5736);
      \draw [CtpText, ->, very thick] (n_3_53_13_5736) to[out=270,in=90] (n_C2_3_53_13_5736);
\end{tikzpicture}}}
  \caption{Flow chart for $\mathcal{D}_3^{9, 37, 7}$ combining EHP and Leibniz rule.}
  \label{fig:flowchart3}
\end{figure}
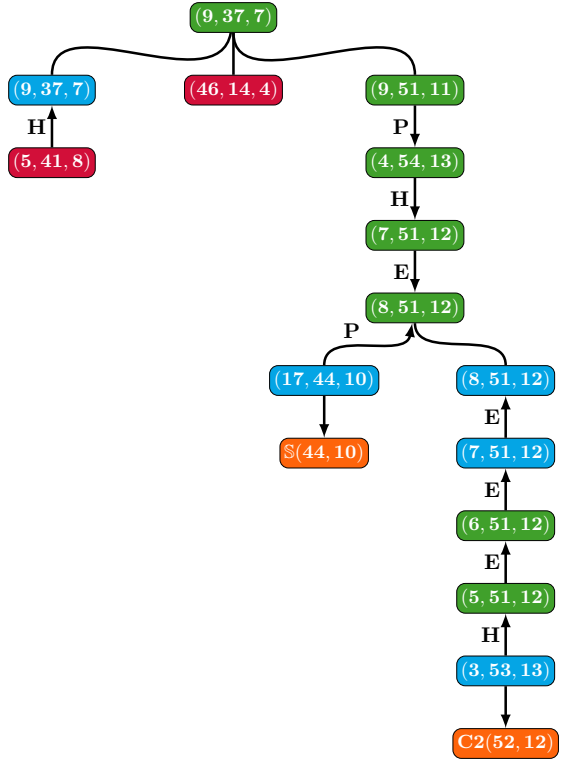
\end{example}

\clearpage

\section{Results}
\label{sec:results}

\noindent
This section presents our results: unstable Adams charts for the spheres $S^n$ with $2 \leq n \leq 50$, displaying proven differentials as well as the remaining possible unknown differentials on each page, and the resulting upper and lower bounds on the orders of the $2$-primary unstable homotopy groups, tabulated in \Cref{app:bounds}. A guide to reading our Adams charts is given in \Cref{app:charts}, and the legend for the flow charts we use to present differential proofs is given in \Crefrange{fig:legend-terminal}{fig:legend-products}.

\subsection{Uncertainties}
If the $d_r$-differential on some tridegree is not fully determined, then we cannot properly evaluate relations that land in that tridegree on the $E_{r + 1}$-page. We are then unable to make general deductions about the $d_{r + 1}$-differential on that tridegree. As we continue turning the page, ignoring these tridegrees creates a growing family of uncertain differentials around that original uncertain tridegree. In order to avoid this, we implement a much simpler element-wise solver which imposes the following logic, for $\mathrm{F}$ ranging over the maps of \Cref{prop:ehp} and \Cref{thm:mah} and over composition with a fixed class:
\begin{enumerate}
  \item If $x$ is a cycle then $\mathrm{F}(x)$ is a cycle
  \item If $\mathrm{F}(x)$ is not a boundary then $x$ is not a boundary
\end{enumerate}
This solver is applied at the end of the computation on each page and to tridegrees which were skipped by the main computation due to an uncertain differential value on an earlier page.
\subsection{Differentials by contradiction}
A handful of differentials can be resolved via proof by contradiction. In future work we will systematically apply this approach throughout all tridegrees. In this project we only compute 17 differentials this way. These particular differentials were chosen because leaving them unresolved makes our Adams charts much harder to read. The proofs are all similar and are easy to check by hand.

\begin{table}[H]
\centering
\caption{Tridegrees $(n, s, f)$ with $d_3$-differentials resolved by proof by contradiction.}
\label{tab:contradiction-tridegrees}
\begin{tabular}{ccccc}
$(5, 47, 10)$ & $(6, 34, 10)$ & $(6, 37, 10)$ & $(6, 42, 14)$ & $(6, 43, 10)$ \\
$(6, 45, 14)$ & $(6, 48, 7)$ & $(7, 47, 16)$ & $(10, 38, 10)$ & $(10, 41, 10)$ \\
$(10, 46, 14)$ & $(10, 47, 10)$ & $(10, 49, 14)$ & $(11, 51, 10)$ & $(13, 51, 9)$ \\
$(14, 36, 7)$ & $(14, 47, 11)$ \\
\end{tabular}
\end{table}

\begin{example}
  There is a possible $d_3$-differential on the unique generator $x$ of $\mathcal{U}_3^{6, 37, 10}.$ If this differential were zero then $x$ would survive to $\mathcal{U}_4^{6, 37, 10}$ where we would have a $d_4$-differential $d_4(H(x)) = H(y)$ on $\mathcal{U}_4^{11, 32, 9}$ where $y$ is the unique generator of $\mathcal{U}_4^{6, 36, 14}.$ Naturality of $H$ would then imply that $d_4(x) = y$ but this is impossible since $y \circ h_1$ is not a boundary for degree reasons so $y$ cannot be a boundary. Therefore our initial assumption that the $d_3$-differential on $\mathcal{U}_3^{6, 37, 10}$ is trivial must be false so $d_3(x)$ is equal to the unique generator of $\mathcal{U}_3^{6, 36, 13}.$
\end{example}

\subsection{Why graphs}
Clicking a differential in our interactive Adams charts opens the flow chart of its proof. Zero differentials are not drawn in the charts, but their flow charts can be queried directly with the \texttt{why.py} script in~\cite{ehpreprint_data}.

\newpage

\subsection{Unstable Adams charts}
\label{app:charts}
\noindent The results of our computation are displayed graphically as Adams charts in the GitHub repository~\cite{ehpreprint_data}. The horizontal axis is the stem, the vertical axis is the Adams filtration, and each node in a bidegree is an $\mathbb{F}_2$-basis element of the corresponding bigraded $\operatorname{Ext}$ group. 
\medskip
\noindent Recall that we write $[\lambda_I]$ for the class whose cycle representative in the homology of the lambda algebra has leading term $\lambda_I$. Vertical lines correspond to right multiplication by $[\lambda_0]$, given by unstable algebraic Yoneda composition, carrying a class $[\lambda_I]$ to $[\lambda_I] \circ [\lambda_0]$. The lines of slope $1$ and slope $1/3$ correspond to algebraic right multiplication by $[\lambda_1]$ and $[\lambda_3]$, respectively (see \Cref{fig:chart-structure}).

\begin{figure}[H]
\centering
\begin{tikzpicture}[
    x=1.4cm, y=1.4cm, 
    line width=0.9pt,
    dot/.style={circle, fill=CtpText, inner sep=0pt, minimum size=4pt},
    every node/.style={font=\small},
  ]
  \draw[CtpSurface1, line width=0.4pt, step=1] (-0.3,-0.3) grid (3.3,1.3);
  \draw[CtpText] (0,0) -- (0,1); 
  \draw[CtpText] (0,0) -- (1,1); 
  \draw[CtpText] (0,0) -- (3,1); 
  \node[dot] at (0,0) {};
  \node[dot] at (0,1) {};
  \node[dot] at (1,1) {};
  \node[dot] at (3,1) {};
  \node[below] at (0,0) {$[\lambda_I]$};
  \node[above] at (0,1) {$[\lambda_I] \circ [\lambda_0]$};
  \node[right] at (1,1) {$[\lambda_I] \circ [\lambda_1]$};
  \node[right] at (3,1) {$[\lambda_I] \circ [\lambda_3]$};
\end{tikzpicture}
\caption{The structural product lines. From a class $[\lambda_I]$, a vertical line records right multiplication by $[\lambda_0]$, a slope-$1$ line records right multiplication by $[\lambda_1]$, and a slope-$1/3$ line records right multiplication by $[\lambda_3]$.}
\label{fig:chart-structure}
\end{figure}
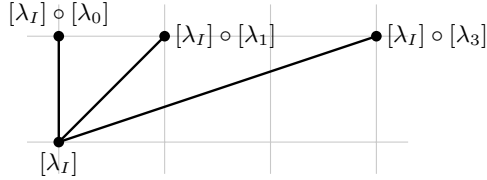

\medskip
\noindent Solid negative-slope lines correspond to unstable Adams differentials, with $d_2$, $d_3$, $d_4$, $d_5$, $d_6$, $d_7$, and $d_8$ drawn in {\color{CtpTeal}teal}, {\color{CtpRed}red}, {\color{CtpGreen}green}, {\color{CtpBlue}navy}, {\color{CtpYellow}yellow}, {\color{CtpPeach}peach}, and {\color{CtpMauve}mauve}, respectively. Dashed negative-slope lines of the same colors indicate possible values of a $d_r$-differential that our algorithm was able neither to prove nor to disprove. If a basis element on the $E_r$-page supports neither a solid line nor a dashed line, then the $d_r$-differential on that element is proved to be zero.

\medskip
\noindent Our graphical calculus does leave some room for ambiguity. For example, the display in \Cref{fig:chart-ambiguity} could indicate that either $d_r(x) \in \{0, y + z\}$ or $d_r(x) \in \{y, z\}$. There are other possibilities as well. In such cases the raw CSV data records the possible values of each differential unambiguously.

\begin{figure}[H]
\centering
\begin{tikzpicture}[
    x=1.4cm, y=1.4cm,
    line width=0.9pt,
    dot/.style={circle, fill=CtpText, inner sep=0pt, minimum size=4pt},
    every node/.style={font=\small},
  ]
  \draw[CtpSurface1, line width=0.4pt, step=1] (-1.4,-0.3) grid (0.4,2.3);
  \draw[CtpTeal, dashed] (0,0) -- (-1.1,2);
  \draw[CtpTeal, dashed] (0,0) -- (-0.9,2);
  \node[dot] at (0,0)    {};
  \node[dot] at (-1.1,2) {};
  \node[dot] at (-0.9,2) {};
  \node[below right] at (0,0)     {$x$};
  \node[above left]  at (-1.1,2)  {$y$};
  \node[above right] at (-0.9,2)  {$z$};
\end{tikzpicture}
\caption{An ambiguous display: the reading may be $d_r(x) \in \{0, y + z\}$ or $d_r(x) \in \{y, z\}$; the raw data resolves which.}
\label{fig:chart-ambiguity}
\end{figure}
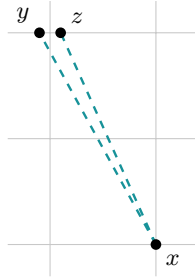
\subsection{Bounds on the orders of the \texorpdfstring{$2$}{2}-primary unstable homotopy groups of spheres}
\label{app:bounds}
\noindent \Crefrange{tab:bounds-true-2-14}{tab:bounds-true-39-50} record the bounds on the orders of the $2$-primary unstable homotopy groups of spheres proven by our computation, and are read as follows. The entry in column $n$ and row $s$ bounds the order of the $2$-primary component of $\pi_{n+s}(S^n)$, recorded through its base-$2$ logarithm: a single entry $e$ marks a homotopy group with no remaining unknown differentials, whose $2$-primary component has order exactly $2^e$. When a homotopy group still admits unknown differentials we can only bracket this order, and the entry is a range $e/\ell$ printed upper bound first: the $2$-primary component of $\pi_{n+s}(S^n)$ then has order $2^m$ for some $\ell \leq m \leq e$, so its order lies between the lower bound $2^\ell$ and the upper bound $2^e$. An entry of $\infty$ marks a homotopy group containing a $\mathbb{Z}$ summand. The {\color{CtpPink}pink staircase boundary} traces the metastable range: the entries lying on or to the right of it, up to the stable line, are in the metastable range $s \leq 3n - 4$, while those to its left are purely unstable. Finally, in the stable range $n \geq s + 2$ every entry in row $s$ records the order of the stable $s$-stem, so rather than reiterate it we print it once, at $n = s + 2$, and draw the repetitions as a vertical line running down each column.

\providecommand{\rngcell}[2]{#1/#2}
\providecommand{\stairlift}{3pt}
\providecommand{\stairshift}{2.5pt}
\providecommand{\stairedge}[1]{\ifcsname stairseen#1\endcsname\else\expandafter\gdef\csname stairseen#1\endcsname{}\tikzmark{#1}\fi}
\providecommand{\inftycell}{{\color{CtpOverlay2}\(\infty\)}}
\providecommand{\stabline}{{\color{CtpOverlay0}\rule[-3pt]{0.5pt}{11pt}}}

\begin{table}[p]
\caption{Bounds on the order of the $2$-primary component of the $s$-stem of $S^n$, recorded as a base-$2$ logarithm (upper/lower), $n = 2, \dots, 14$.}
\label{tab:bounds-true-2-14}
\centering
\footnotesize
\setlength{\tabcolsep}{3pt}
\renewcommand{\arraystretch}{1.08}
\begin{tabular*}{\linewidth}{@{\extracolsep{\fill}} r!{\stairedge{La}\vrule}*{13}{c} @{\stairedge{Ra}}}
\toprule
$s \backslash n$ & 2 & 3 & 4 & 5 & 6 & 7 & 8 & 9 & 10 & 11 & 12 & 13 & 14 \\
\midrule
0 & \inftycell & \stabline & \stabline & \stabline & \stabline & \stabline & \stabline & \stabline & \stabline & \stabline & \stabline & \stabline & \stabline \\
1 & \inftycell & 1 & \stabline & \stabline & \stabline & \stabline & \stabline & \stabline & \stabline & \stabline & \stabline & \stabline & \stabline \\
2 & 1 & 1 & 1 & \stabline & \stabline & \stabline & \stabline & \stabline & \stabline & \stabline & \stabline & \stabline & \stabline \\
3 & \tikzmark{La3}\makebox[0pt]{1} & \tikzmark{Ra3}\makebox[0pt]{2} & \inftycell & 3 & \stabline & \stabline & \stabline & \stabline & \stabline & \stabline & \stabline & \stabline & \stabline \\
4 & \tikzmark{La4}\makebox[0pt]{2} & \tikzmark{Ra4}\makebox[0pt]{1} & 2 & 1 & 0 & \stabline & \stabline & \stabline & \stabline & \stabline & \stabline & \stabline & \stabline \\
5 & \tikzmark{La5}\makebox[0pt]{1} & \tikzmark{Ra5}\makebox[0pt]{1} & 2 & 1 & \inftycell & 0 & \stabline & \stabline & \stabline & \stabline & \stabline & \stabline & \stabline \\
6 & 1 & \tikzmark{La6}\makebox[0pt]{0} & \tikzmark{Ra6}\makebox[0pt]{3} & 1 & 1 & 1 & 1 & \stabline & \stabline & \stabline & \stabline & \stabline & \stabline \\
7 & 0 & \tikzmark{La7}\makebox[0pt]{0} & \tikzmark{Ra7}\makebox[0pt]{0} & 1 & 2 & 3 & \inftycell & 4 & \stabline & \stabline & \stabline & \stabline & \stabline \\
8 & 0 & \tikzmark{La8}\makebox[0pt]{1} & \tikzmark{Ra8}\makebox[0pt]{1} & 1 & 4 & 3 & 4 & 3 & 2 & \stabline & \stabline & \stabline & \stabline \\
9 & 1 & 2 & \tikzmark{La9}\makebox[0pt]{3} & \tikzmark{Ra9}\makebox[0pt]{3} & 3 & 4 & 5 & 4 & \inftycell & 3 & \stabline & \stabline & \stabline \\
10 & 2 & 3 & \tikzmark{La10}\makebox[0pt]{6} & \tikzmark{Ra10}\makebox[0pt]{4} & 4 & 4 & 7 & 4 & 3 & 2 & 1 & \stabline & \stabline \\
11 & 3 & 4 & \tikzmark{La11}\makebox[0pt]{7} & \tikzmark{Ra11}\makebox[0pt]{5} & 5 & 4 & 4 & 4 & 3 & 3 & \inftycell & 3 & \stabline \\
12 & 4 & 2 & 6 & \tikzmark{La12}\makebox[0pt]{3} & \tikzmark{Ra12}\makebox[0pt]{4} & 0 & 0 & 0 & 2 & 1 & 2 & 1 & 0 \\
13 & 2 & 1 & 5 & \tikzmark{La13}\makebox[0pt]{2} & \tikzmark{Ra13}\makebox[0pt]{1} & 1 & 2 & 1 & 1 & 2 & 2 & 1 & \inftycell \\
14 & 1 & 1 & 5 & \tikzmark{La14}\makebox[0pt]{2} & \tikzmark{Ra14}\makebox[0pt]{3} & 5 & 9 & 6 & 5 & 5 & 7 & 5 & 4 \\
15 & 1 & 1 & 1 & 2 & \tikzmark{La15}\makebox[0pt]{3} & \tikzmark{Ra15}\makebox[0pt]{6} & 8 & 7 & 6 & 5 & 5 & 6 & 6 \\
16 & 1 & 2 & 3 & 2 & \tikzmark{La16}\makebox[0pt]{5} & \tikzmark{Ra16}\makebox[0pt]{4} & 7 & 4 & 5 & 1 & 1 & 1 & 4 \\
17 & 2 & 4 & 9 & 4 & \tikzmark{La17}\makebox[0pt]{4} & \tikzmark{Ra17}\makebox[0pt]{4} & 5 & 4 & 3 & 3 & 4 & 4 & 4 \\
18 & 4 & 4 & 10 & 5 & \rngcell{5}{4} & \tikzmark{La18}\makebox[0pt]{\rngcell{4}{3}} & \tikzmark{Ra18}\makebox[0pt]{\rngcell{7}{6}} & 4 & 5 & 6 & 10 & 7 & 7 \\
19 & 4 & 3 & 7 & 4 & \rngcell{8}{6} & \tikzmark{La19}\makebox[0pt]{\rngcell{4}{3}} & \tikzmark{Ra19}\makebox[0pt]{\rngcell{4}{3}} & 4 & 4 & 6 & 8 & 6 & 6 \\
20 & 3 & 2 & 6 & 3 & \rngcell{7}{6} & \tikzmark{La20}\makebox[0pt]{3} & \tikzmark{Ra20}\makebox[0pt]{3} & 3 & 6 & 5 & 8 & 6 & 7 \\
21 & 2 & 1 & \rngcell{5}{4} & 2 & 1 & 2 & \tikzmark{La21}\makebox[0pt]{5} & \tikzmark{Ra21}\makebox[0pt]{3} & 3 & 4 & \rngcell{5}{4} & 5 & 4 \\
22 & 1 & 1 & \rngcell{5}{4} & 3 & 5 & 6 & \tikzmark{La22}\makebox[0pt]{12} & \tikzmark{Ra22}\makebox[0pt]{7} & \rngcell{6}{5} & \rngcell{6}{5} & \rngcell{9}{7} & 6 & 7 \\
23 & 1 & 2 & 5 & 5 & \rngcell{11}{10} & \rngcell{10}{9} & \tikzmark{La23}\makebox[0pt]{\rngcell{13}{12}} & \tikzmark{Ra23}\makebox[0pt]{\rngcell{11}{10}} & \rngcell{11}{10} & \rngcell{9}{8} & \rngcell{9}{8} & 9 & 9 \\
24 & 2 & 3 & 5 & 4 & \rngcell{9}{8} & \rngcell{7}{5} & \rngcell{12}{8} & \tikzmark{La24}\makebox[0pt]{\rngcell{7}{6}} & \tikzmark{Ra24}\makebox[0pt]{10} & 5 & 4 & 4 & 7 \\
25 & 3 & 5 & 11 & 7 & 10 & \rngcell{8}{7} & \rngcell{15}{12} & \tikzmark{La25}\makebox[0pt]{8} & \tikzmark{Ra25}\makebox[0pt]{6} & 6 & 7 & 5 & 5 \\
26 & 5 & 5 & \rngcell{15}{14} & 8 & 11 & 9 & 14 & \tikzmark{La26}\makebox[0pt]{9} & \tikzmark{Ra26}\makebox[0pt]{8} & 7 & \rngcell{13}{11} & 7 & 5 \\
27 & 5 & 5 & \rngcell{12}{10} & 7 & 10 & \rngcell{7}{6} & \rngcell{10}{9} & 7 & \tikzmark{La27}\makebox[0pt]{8} & \tikzmark{Ra27}\makebox[0pt]{8} & \rngcell{9}{7} & 7 & 5 \\
28 & 5 & 4 & \rngcell{12}{11} & 5 & \rngcell{9}{8} & \rngcell{4}{3} & \rngcell{7}{6} & 5 & \tikzmark{La28}\makebox[0pt]{7} & \tikzmark{Ra28}\makebox[0pt]{6} & 8 & 4 & 6 \\
29 & 4 & 3 & 12 & 6 & \rngcell{6}{5} & 7 & 14 & 8 & \tikzmark{La29}\makebox[0pt]{9} & \tikzmark{Ra29}\makebox[0pt]{6} & \rngcell{8}{7} & \rngcell{5}{4} & \rngcell{3}{2} \\
30 & 3 & 3 & \rngcell{10}{9} & 6 & 8 & 9 & 20 & 10 & \rngcell{10}{9} & \tikzmark{La30}\makebox[0pt]{7} & \tikzmark{Ra30}\makebox[0pt]{\rngcell{9}{8}} & \rngcell{7}{6} & \rngcell{9}{7} \\
31 & 3 & 4 & \rngcell{8}{7} & 5 & 10 & 7 & 13 & 8 & \rngcell{9}{8} & \tikzmark{La31}\makebox[0pt]{5} & \tikzmark{Ra31}\makebox[0pt]{7} & 6 & \rngcell{10}{9} \\
32 & 4 & 4 & 11 & 5 & 11 & 6 & 12 & 6 & \rngcell{10}{9} & \tikzmark{La32}\makebox[0pt]{4} & \tikzmark{Ra32}\makebox[0pt]{5} & 5 & 8 \\
33 & 4 & 3 & 12 & 4 & 8 & 6 & 11 & 6 & \rngcell{5}{4} & 7 & \tikzmark{La33}\makebox[0pt]{\rngcell{8}{7}} & \tikzmark{Ra33}\makebox[0pt]{8} & 9 \\
34 & 3 & 4 & 11 & 5 & \rngcell{9}{8} & 6 & 9 & 6 & 6 & 7 & \tikzmark{La34}\makebox[0pt]{\rngcell{14}{12}} & \tikzmark{Ra34}\makebox[0pt]{8} & 9 \\
35 & 4 & 4 & 10 & 5 & \rngcell{11}{9} & \rngcell{6}{5} & \rngcell{7}{6} & 6 & 10 & 7 & \tikzmark{La35}\makebox[0pt]{\rngcell{10}{9}} & \tikzmark{Ra35}\makebox[0pt]{8} & 8 \\
36 & 4 & 3 & 9 & 4 & \rngcell{6}{5} & \rngcell{2}{1} & \rngcell{5}{3} & 2 & 7 & 4 & 7 & \tikzmark{La36}\makebox[0pt]{4} & \tikzmark{Ra36}\makebox[0pt]{\rngcell{9}{8}} \\
37 & 3 & 2 & 8 & 4 & 4 & 3 & \rngcell{12}{9} & 3 & 5 & 4 & \rngcell{9}{8} & \tikzmark{La37}\makebox[0pt]{5} & \tikzmark{Ra37}\makebox[0pt]{\rngcell{5}{4}} \\
38 & 2 & 1 & \rngcell{7}{6} & 4 & 9 & 8 & \rngcell{19}{17} & 9 & 11 & 9 & \rngcell{13}{12} & \tikzmark{La38}\makebox[0pt]{10} & \tikzmark{Ra38}\makebox[0pt]{9} \\
39 & 1 & 2 & \rngcell{4}{3} & 4 & 10 & 10 & \rngcell{18}{17} & 11 & 14 & 8 & 9 & 8 & 9 \\
40 & 2 & 3 & 6 & 4 & \rngcell{10}{9} & \rngcell{10}{9} & \rngcell{19}{17} & 10 & 12 & 4 & 5 & 4 & 7 \\
41 & 3 & 4 & 12 & 6 & \rngcell{8}{7} & \rngcell{9}{8} & \rngcell{18}{16} & 9 & \rngcell{9}{8} & 6 & \rngcell{11}{10} & 5 & \rngcell{7}{6} \\
42 & 4 & 6 & 16 & 9 & 11 & 10 & \rngcell{19}{17} & 10 & \rngcell{10}{9} & \rngcell{9}{8} & \rngcell{17}{14} & \rngcell{7}{6} & \rngcell{8}{7} \\
43 & 6 & 7 & \rngcell{17}{16} & 9 & 14 & 9 & \rngcell{14}{13} & 9 & \rngcell{12}{11} & \rngcell{7}{6} & \rngcell{12}{10} & \rngcell{6}{5} & \rngcell{8}{7} \\
44 & 7 & 5 & \rngcell{14}{12} & 5 & 8 & 4 & 11 & 5 & \rngcell{12}{10} & \rngcell{5}{4} & \rngcell{10}{9} & 5 & \rngcell{9}{8} \\
45 & 5 & 3 & \rngcell{13}{12} & 4 & \rngcell{3}{2} & 3 & 13 & 5 & \rngcell{7}{6} & \rngcell{5}{4} & \rngcell{10}{8} & 6 & \rngcell{6}{5} \\
46 & 3 & 1 & 10 & 5 & \rngcell{5}{4} & 6 & 15 & 9 & \rngcell{12}{11} & 8 & \rngcell{13}{11} & \rngcell{11}{10} & \rngcell{13}{11} \\
47 & 1 & 1 & 5 & 5 & \rngcell{10}{9} & \rngcell{10}{9} & \rngcell{16}{15} & 12 & \rngcell{16}{14} & \rngcell{11}{10} & \rngcell{12}{9} & \rngcell{13}{11} & \rngcell{19}{18} \\
48 & 1 & 4 & 7 & \rngcell{6}{5} & \rngcell{12}{11} & \rngcell{11}{10} & \rngcell{17}{16} & \rngcell{11}{10} & \rngcell{12}{10} & \rngcell{8}{6} & \rngcell{11}{9} & \rngcell{8}{7} & 15 \\
49 & 4 & 6 & \rngcell{12}{11} & \rngcell{7}{5} & \rngcell{13}{10} & \rngcell{8}{7} & \rngcell{15}{14} & \rngcell{8}{6} & \rngcell{8}{5} & \rngcell{8}{5} & \rngcell{16}{12} & \rngcell{9}{7} & \rngcell{11}{8} \\
50 & \rngcell{6}{5} & 7 & \rngcell{17}{14} & \rngcell{8}{6} & \rngcell{12}{7} & \rngcell{5}{4} & \rngcell{9}{8} & \rngcell{5}{4} & \rngcell{10}{8} & \rngcell{11}{9} & \rngcell{23}{14} & \rngcell{13}{11} & \rngcell{15}{12} \\
\bottomrule
\end{tabular*}
\begin{tikzpicture}[overlay, remember picture]
  \coordinate (Pa3) at ($(pic cs:La3)!0.5!(pic cs:Ra3)+(\stairshift,\stairlift)$);
  \coordinate (Pa4) at ($(pic cs:La4)!0.5!(pic cs:Ra4)+(\stairshift,\stairlift)$);
  \coordinate (Pa5) at ($(pic cs:La5)!0.5!(pic cs:Ra5)+(\stairshift,\stairlift)$);
  \coordinate (Pa6) at ($(pic cs:La6)!0.5!(pic cs:Ra6)+(\stairshift,\stairlift)$);
  \coordinate (Pa7) at ($(pic cs:La7)!0.5!(pic cs:Ra7)+(\stairshift,\stairlift)$);
  \coordinate (Pa8) at ($(pic cs:La8)!0.5!(pic cs:Ra8)+(\stairshift,\stairlift)$);
  \coordinate (Pa9) at ($(pic cs:La9)!0.5!(pic cs:Ra9)+(\stairshift,\stairlift)$);
  \coordinate (Pa10) at ($(pic cs:La10)!0.5!(pic cs:Ra10)+(\stairshift,\stairlift)$);
  \coordinate (Pa11) at ($(pic cs:La11)!0.5!(pic cs:Ra11)+(\stairshift,\stairlift)$);
  \coordinate (Pa12) at ($(pic cs:La12)!0.5!(pic cs:Ra12)+(\stairshift,\stairlift)$);
  \coordinate (Pa13) at ($(pic cs:La13)!0.5!(pic cs:Ra13)+(\stairshift,\stairlift)$);
  \coordinate (Pa14) at ($(pic cs:La14)!0.5!(pic cs:Ra14)+(\stairshift,\stairlift)$);
  \coordinate (Pa15) at ($(pic cs:La15)!0.5!(pic cs:Ra15)+(\stairshift,\stairlift)$);
  \coordinate (Pa16) at ($(pic cs:La16)!0.5!(pic cs:Ra16)+(\stairshift,\stairlift)$);
  \coordinate (Pa17) at ($(pic cs:La17)!0.5!(pic cs:Ra17)+(\stairshift,\stairlift)$);
  \coordinate (Pa18) at ($(pic cs:La18)!0.5!(pic cs:Ra18)+(\stairshift,\stairlift)$);
  \coordinate (Pa19) at ($(pic cs:La19)!0.5!(pic cs:Ra19)+(\stairshift,\stairlift)$);
  \coordinate (Pa20) at ($(pic cs:La20)!0.5!(pic cs:Ra20)+(\stairshift,\stairlift)$);
  \coordinate (Pa21) at ($(pic cs:La21)!0.5!(pic cs:Ra21)+(\stairshift,\stairlift)$);
  \coordinate (Pa22) at ($(pic cs:La22)!0.5!(pic cs:Ra22)+(\stairshift,\stairlift)$);
  \coordinate (Pa23) at ($(pic cs:La23)!0.5!(pic cs:Ra23)+(\stairshift,\stairlift)$);
  \coordinate (Pa24) at ($(pic cs:La24)!0.5!(pic cs:Ra24)+(\stairshift,\stairlift)$);
  \coordinate (Pa25) at ($(pic cs:La25)!0.5!(pic cs:Ra25)+(\stairshift,\stairlift)$);
  \coordinate (Pa26) at ($(pic cs:La26)!0.5!(pic cs:Ra26)+(\stairshift,\stairlift)$);
  \coordinate (Pa27) at ($(pic cs:La27)!0.5!(pic cs:Ra27)+(\stairshift,\stairlift)$);
  \coordinate (Pa28) at ($(pic cs:La28)!0.5!(pic cs:Ra28)+(\stairshift,\stairlift)$);
  \coordinate (Pa29) at ($(pic cs:La29)!0.5!(pic cs:Ra29)+(\stairshift,\stairlift)$);
  \coordinate (Pa30) at ($(pic cs:La30)!0.5!(pic cs:Ra30)+(\stairshift,\stairlift)$);
  \coordinate (Pa31) at ($(pic cs:La31)!0.5!(pic cs:Ra31)+(\stairshift,\stairlift)$);
  \coordinate (Pa32) at ($(pic cs:La32)!0.5!(pic cs:Ra32)+(\stairshift,\stairlift)$);
  \coordinate (Pa33) at ($(pic cs:La33)!0.5!(pic cs:Ra33)+(\stairshift,\stairlift)$);
  \coordinate (Pa34) at ($(pic cs:La34)!0.5!(pic cs:Ra34)+(\stairshift,\stairlift)$);
  \coordinate (Pa35) at ($(pic cs:La35)!0.5!(pic cs:Ra35)+(\stairshift,\stairlift)$);
  \coordinate (Pa36) at ($(pic cs:La36)!0.5!(pic cs:Ra36)+(\stairshift,\stairlift)$);
  \coordinate (Pa37) at ($(pic cs:La37)!0.5!(pic cs:Ra37)+(\stairshift,\stairlift)$);
  \coordinate (Pa38) at ($(pic cs:La38)!0.5!(pic cs:Ra38)+(\stairshift,\stairlift)$);
  \draw[CtpPink, line width=0.9pt, line join=miter, line cap=round] let \p1=(Pa3), \p2=(Pa4), \p3=(pic cs:La) in
    (\x3, {\y1+(\y1-\y2)/2}) -- (\x1, {\y1+(\y1-\y2)/2}) -- (\x1,\y1);
  \draw[CtpPink, line width=0.9pt, line join=miter, line cap=round] let \p1=(Pa3), \p2=(Pa4) in
    (\x1,\y1) -- (\x1,{(\y1+\y2)/2}) -- (\x2,{(\y1+\y2)/2}) -- (\x2,\y2);
  \draw[CtpPink, line width=0.9pt, line join=miter, line cap=round] let \p1=(Pa4), \p2=(Pa5) in
    (\x1,\y1) -- (\x1,{(\y1+\y2)/2}) -- (\x2,{(\y1+\y2)/2}) -- (\x2,\y2);
  \draw[CtpPink, line width=0.9pt, line join=miter, line cap=round] let \p1=(Pa5), \p2=(Pa6) in
    (\x1,\y1) -- (\x1,{(\y1+\y2)/2}) -- (\x2,{(\y1+\y2)/2}) -- (\x2,\y2);
  \draw[CtpPink, line width=0.9pt, line join=miter, line cap=round] let \p1=(Pa6), \p2=(Pa7) in
    (\x1,\y1) -- (\x1,{(\y1+\y2)/2}) -- (\x2,{(\y1+\y2)/2}) -- (\x2,\y2);
  \draw[CtpPink, line width=0.9pt, line join=miter, line cap=round] let \p1=(Pa7), \p2=(Pa8) in
    (\x1,\y1) -- (\x1,{(\y1+\y2)/2}) -- (\x2,{(\y1+\y2)/2}) -- (\x2,\y2);
  \draw[CtpPink, line width=0.9pt, line join=miter, line cap=round] let \p1=(Pa8), \p2=(Pa9) in
    (\x1,\y1) -- (\x1,{(\y1+\y2)/2}) -- (\x2,{(\y1+\y2)/2}) -- (\x2,\y2);
  \draw[CtpPink, line width=0.9pt, line join=miter, line cap=round] let \p1=(Pa9), \p2=(Pa10) in
    (\x1,\y1) -- (\x1,{(\y1+\y2)/2}) -- (\x2,{(\y1+\y2)/2}) -- (\x2,\y2);
  \draw[CtpPink, line width=0.9pt, line join=miter, line cap=round] let \p1=(Pa10), \p2=(Pa11) in
    (\x1,\y1) -- (\x1,{(\y1+\y2)/2}) -- (\x2,{(\y1+\y2)/2}) -- (\x2,\y2);
  \draw[CtpPink, line width=0.9pt, line join=miter, line cap=round] let \p1=(Pa11), \p2=(Pa12) in
    (\x1,\y1) -- (\x1,{(\y1+\y2)/2}) -- (\x2,{(\y1+\y2)/2}) -- (\x2,\y2);
  \draw[CtpPink, line width=0.9pt, line join=miter, line cap=round] let \p1=(Pa12), \p2=(Pa13) in
    (\x1,\y1) -- (\x1,{(\y1+\y2)/2}) -- (\x2,{(\y1+\y2)/2}) -- (\x2,\y2);
  \draw[CtpPink, line width=0.9pt, line join=miter, line cap=round] let \p1=(Pa13), \p2=(Pa14) in
    (\x1,\y1) -- (\x1,{(\y1+\y2)/2}) -- (\x2,{(\y1+\y2)/2}) -- (\x2,\y2);
  \draw[CtpPink, line width=0.9pt, line join=miter, line cap=round] let \p1=(Pa14), \p2=(Pa15) in
    (\x1,\y1) -- (\x1,{(\y1+\y2)/2}) -- (\x2,{(\y1+\y2)/2}) -- (\x2,\y2);
  \draw[CtpPink, line width=0.9pt, line join=miter, line cap=round] let \p1=(Pa15), \p2=(Pa16) in
    (\x1,\y1) -- (\x1,{(\y1+\y2)/2}) -- (\x2,{(\y1+\y2)/2}) -- (\x2,\y2);
  \draw[CtpPink, line width=0.9pt, line join=miter, line cap=round] let \p1=(Pa16), \p2=(Pa17) in
    (\x1,\y1) -- (\x1,{(\y1+\y2)/2}) -- (\x2,{(\y1+\y2)/2}) -- (\x2,\y2);
  \draw[CtpPink, line width=0.9pt, line join=miter, line cap=round] let \p1=(Pa17), \p2=(Pa18) in
    (\x1,\y1) -- (\x1,{(\y1+\y2)/2}) -- (\x2,{(\y1+\y2)/2}) -- (\x2,\y2);
  \draw[CtpPink, line width=0.9pt, line join=miter, line cap=round] let \p1=(Pa18), \p2=(Pa19) in
    (\x1,\y1) -- (\x1,{(\y1+\y2)/2}) -- (\x2,{(\y1+\y2)/2}) -- (\x2,\y2);
  \draw[CtpPink, line width=0.9pt, line join=miter, line cap=round] let \p1=(Pa19), \p2=(Pa20) in
    (\x1,\y1) -- (\x1,{(\y1+\y2)/2}) -- (\x2,{(\y1+\y2)/2}) -- (\x2,\y2);
  \draw[CtpPink, line width=0.9pt, line join=miter, line cap=round] let \p1=(Pa20), \p2=(Pa21) in
    (\x1,\y1) -- (\x1,{(\y1+\y2)/2}) -- (\x2,{(\y1+\y2)/2}) -- (\x2,\y2);
  \draw[CtpPink, line width=0.9pt, line join=miter, line cap=round] let \p1=(Pa21), \p2=(Pa22) in
    (\x1,\y1) -- (\x1,{(\y1+\y2)/2}) -- (\x2,{(\y1+\y2)/2}) -- (\x2,\y2);
  \draw[CtpPink, line width=0.9pt, line join=miter, line cap=round] let \p1=(Pa22), \p2=(Pa23) in
    (\x1,\y1) -- (\x1,{(\y1+\y2)/2}) -- (\x2,{(\y1+\y2)/2}) -- (\x2,\y2);
  \draw[CtpPink, line width=0.9pt, line join=miter, line cap=round] let \p1=(Pa23), \p2=(Pa24) in
    (\x1,\y1) -- (\x1,{(\y1+\y2)/2}) -- (\x2,{(\y1+\y2)/2}) -- (\x2,\y2);
  \draw[CtpPink, line width=0.9pt, line join=miter, line cap=round] let \p1=(Pa24), \p2=(Pa25) in
    (\x1,\y1) -- (\x1,{(\y1+\y2)/2}) -- (\x2,{(\y1+\y2)/2}) -- (\x2,\y2);
  \draw[CtpPink, line width=0.9pt, line join=miter, line cap=round] let \p1=(Pa25), \p2=(Pa26) in
    (\x1,\y1) -- (\x1,{(\y1+\y2)/2}) -- (\x2,{(\y1+\y2)/2}) -- (\x2,\y2);
  \draw[CtpPink, line width=0.9pt, line join=miter, line cap=round] let \p1=(Pa26), \p2=(Pa27) in
    (\x1,\y1) -- (\x1,{(\y1+\y2)/2}) -- (\x2,{(\y1+\y2)/2}) -- (\x2,\y2);
  \draw[CtpPink, line width=0.9pt, line join=miter, line cap=round] let \p1=(Pa27), \p2=(Pa28) in
    (\x1,\y1) -- (\x1,{(\y1+\y2)/2}) -- (\x2,{(\y1+\y2)/2}) -- (\x2,\y2);
  \draw[CtpPink, line width=0.9pt, line join=miter, line cap=round] let \p1=(Pa28), \p2=(Pa29) in
    (\x1,\y1) -- (\x1,{(\y1+\y2)/2}) -- (\x2,{(\y1+\y2)/2}) -- (\x2,\y2);
  \draw[CtpPink, line width=0.9pt, line join=miter, line cap=round] let \p1=(Pa29), \p2=(Pa30) in
    (\x1,\y1) -- (\x1,{(\y1+\y2)/2}) -- (\x2,{(\y1+\y2)/2}) -- (\x2,\y2);
  \draw[CtpPink, line width=0.9pt, line join=miter, line cap=round] let \p1=(Pa30), \p2=(Pa31) in
    (\x1,\y1) -- (\x1,{(\y1+\y2)/2}) -- (\x2,{(\y1+\y2)/2}) -- (\x2,\y2);
  \draw[CtpPink, line width=0.9pt, line join=miter, line cap=round] let \p1=(Pa31), \p2=(Pa32) in
    (\x1,\y1) -- (\x1,{(\y1+\y2)/2}) -- (\x2,{(\y1+\y2)/2}) -- (\x2,\y2);
  \draw[CtpPink, line width=0.9pt, line join=miter, line cap=round] let \p1=(Pa32), \p2=(Pa33) in
    (\x1,\y1) -- (\x1,{(\y1+\y2)/2}) -- (\x2,{(\y1+\y2)/2}) -- (\x2,\y2);
  \draw[CtpPink, line width=0.9pt, line join=miter, line cap=round] let \p1=(Pa33), \p2=(Pa34) in
    (\x1,\y1) -- (\x1,{(\y1+\y2)/2}) -- (\x2,{(\y1+\y2)/2}) -- (\x2,\y2);
  \draw[CtpPink, line width=0.9pt, line join=miter, line cap=round] let \p1=(Pa34), \p2=(Pa35) in
    (\x1,\y1) -- (\x1,{(\y1+\y2)/2}) -- (\x2,{(\y1+\y2)/2}) -- (\x2,\y2);
  \draw[CtpPink, line width=0.9pt, line join=miter, line cap=round] let \p1=(Pa35), \p2=(Pa36) in
    (\x1,\y1) -- (\x1,{(\y1+\y2)/2}) -- (\x2,{(\y1+\y2)/2}) -- (\x2,\y2);
  \draw[CtpPink, line width=0.9pt, line join=miter, line cap=round] let \p1=(Pa36), \p2=(Pa37) in
    (\x1,\y1) -- (\x1,{(\y1+\y2)/2}) -- (\x2,{(\y1+\y2)/2}) -- (\x2,\y2);
  \draw[CtpPink, line width=0.9pt, line join=miter, line cap=round] let \p1=(Pa37), \p2=(Pa38) in
    (\x1,\y1) -- (\x1,{(\y1+\y2)/2}) -- (\x2,{(\y1+\y2)/2}) -- (\x2,\y2);
  \draw[CtpPink, line width=0.9pt, line join=miter, line cap=round] let \p1=(Pa37), \p2=(Pa38), \p3=(pic cs:Ra) in
    (\x2,\y2) -- (\x2, {\y2-(\y1-\y2)/2}) -- (\x3, {\y2-(\y1-\y2)/2});
\end{tikzpicture}
\end{table}

\begin{table}[p]
\caption{Bounds on the order of the $2$-primary component of the $s$-stem of $S^n$, recorded as a base-$2$ logarithm (upper/lower), $n = 15, \dots, 26$.}
\label{tab:bounds-true-15-26}
\centering
\footnotesize
\setlength{\tabcolsep}{3pt}
\renewcommand{\arraystretch}{1.08}
\begin{tabular*}{\linewidth}{@{\extracolsep{\fill}} r!{\stairedge{Lb}\vrule}*{12}{c} @{\stairedge{Rb}}}
\toprule
$s \backslash n$ & 15 & 16 & 17 & 18 & 19 & 20 & 21 & 22 & 23 & 24 & 25 & 26 \\
\midrule
0 & \stabline & \stabline & \stabline & \stabline & \stabline & \stabline & \stabline & \stabline & \stabline & \stabline & \stabline & \stabline \\
1 & \stabline & \stabline & \stabline & \stabline & \stabline & \stabline & \stabline & \stabline & \stabline & \stabline & \stabline & \stabline \\
2 & \stabline & \stabline & \stabline & \stabline & \stabline & \stabline & \stabline & \stabline & \stabline & \stabline & \stabline & \stabline \\
3 & \stabline & \stabline & \stabline & \stabline & \stabline & \stabline & \stabline & \stabline & \stabline & \stabline & \stabline & \stabline \\
4 & \stabline & \stabline & \stabline & \stabline & \stabline & \stabline & \stabline & \stabline & \stabline & \stabline & \stabline & \stabline \\
5 & \stabline & \stabline & \stabline & \stabline & \stabline & \stabline & \stabline & \stabline & \stabline & \stabline & \stabline & \stabline \\
6 & \stabline & \stabline & \stabline & \stabline & \stabline & \stabline & \stabline & \stabline & \stabline & \stabline & \stabline & \stabline \\
7 & \stabline & \stabline & \stabline & \stabline & \stabline & \stabline & \stabline & \stabline & \stabline & \stabline & \stabline & \stabline \\
8 & \stabline & \stabline & \stabline & \stabline & \stabline & \stabline & \stabline & \stabline & \stabline & \stabline & \stabline & \stabline \\
9 & \stabline & \stabline & \stabline & \stabline & \stabline & \stabline & \stabline & \stabline & \stabline & \stabline & \stabline & \stabline \\
10 & \stabline & \stabline & \stabline & \stabline & \stabline & \stabline & \stabline & \stabline & \stabline & \stabline & \stabline & \stabline \\
11 & \stabline & \stabline & \stabline & \stabline & \stabline & \stabline & \stabline & \stabline & \stabline & \stabline & \stabline & \stabline \\
12 & \stabline & \stabline & \stabline & \stabline & \stabline & \stabline & \stabline & \stabline & \stabline & \stabline & \stabline & \stabline \\
13 & 0 & \stabline & \stabline & \stabline & \stabline & \stabline & \stabline & \stabline & \stabline & \stabline & \stabline & \stabline \\
14 & 3 & 2 & \stabline & \stabline & \stabline & \stabline & \stabline & \stabline & \stabline & \stabline & \stabline & \stabline \\
15 & 6 & \inftycell & 6 & \stabline & \stabline & \stabline & \stabline & \stabline & \stabline & \stabline & \stabline & \stabline \\
16 & 3 & 4 & 3 & 2 & \stabline & \stabline & \stabline & \stabline & \stabline & \stabline & \stabline & \stabline \\
17 & 5 & 6 & 5 & \inftycell & 4 & \stabline & \stabline & \stabline & \stabline & \stabline & \stabline & \stabline \\
18 & 7 & 10 & 7 & 6 & 5 & 4 & \stabline & \stabline & \stabline & \stabline & \stabline & \stabline \\
19 & 5 & 5 & 5 & 4 & 4 & \inftycell & 4 & \stabline & \stabline & \stabline & \stabline & \stabline \\
20 & 3 & 3 & 3 & 5 & 4 & 5 & 4 & 3 & \stabline & \stabline & \stabline & \stabline \\
21 & 3 & 4 & 3 & 3 & 4 & 4 & 3 & \inftycell & 2 & \stabline & \stabline & \stabline \\
22 & 7 & 11 & 7 & 6 & 6 & 8 & 5 & 4 & 3 & 2 & \stabline & \stabline \\
23 & 11 & 13 & 11 & 10 & 9 & 9 & 9 & 8 & 8 & \inftycell & 8 & \stabline \\
24 & 6 & 9 & 6 & 7 & 3 & 2 & 2 & 4 & 3 & 4 & 3 & 2 \\
25 & 6 & 7 & 6 & 4 & 3 & 2 & 2 & 2 & 3 & 4 & 3 & \inftycell \\
26 & 5 & 8 & 5 & 5 & 5 & 8 & 5 & 5 & 5 & 8 & 5 & 4 \\
27 & 3 & 3 & 3 & 3 & 5 & 7 & 5 & 5 & 4 & 4 & 4 & 3 \\
28 & 1 & 1 & 1 & 4 & 3 & 6 & 4 & 5 & 1 & 1 & 1 & 3 \\
29 & \rngcell{3}{2} & \rngcell{5}{4} & 2 & 2 & 3 & 4 & 4 & 3 & 2 & 2 & 1 & 1 \\
30 & \rngcell{9}{6} & \rngcell{15}{12} & 8 & 6 & 6 & 9 & 6 & 7 & 7 & \rngcell{10}{9} & 6 & 5 \\
31 & \rngcell{11}{9} & \rngcell{13}{11} & 11 & 10 & 8 & 8 & 8 & 8 & 10 & \rngcell{12}{11} & 10 & 9 \\
32 & \rngcell{8}{7} & \rngcell{12}{10} & 8 & 11 & 6 & 4 & 4 & 7 & 6 & 9 & 6 & 7 \\
33 & \rngcell{9}{8} & \rngcell{13}{11} & 9 & 7 & 7 & 7 & 5 & 5 & 6 & 7 & 6 & 5 \\
34 & \rngcell{9}{8} & \rngcell{13}{11} & \rngcell{9}{8} & 9 & 8 & \rngcell{14}{12} & 8 & 6 & 6 & 9 & 6 & 7 \\
35 & \rngcell{9}{7} & \rngcell{12}{9} & \rngcell{9}{8} & 10 & 10 & \rngcell{11}{9} & 9 & 7 & 5 & 5 & 5 & 5 \\
36 & \rngcell{5}{4} & \rngcell{7}{6} & 5 & 7 & 6 & 8 & 4 & 6 & 1 & 1 & 1 & 4 \\
37 & 7 & 9 & 7 & 7 & 5 & 7 & 4 & 3 & 3 & 5 & 3 & 3 \\
38 & 10 & 18 & 10 & 8 & 6 & 8 & 6 & \rngcell{8}{7} & \rngcell{8}{7} & \rngcell{14}{13} & 8 & \rngcell{7}{6} \\
39 & 9 & 11 & 9 & 7 & 7 & 9 & 8 & \rngcell{11}{10} & \rngcell{12}{11} & \rngcell{14}{13} & 12 & \rngcell{12}{11} \\
40 & 6 & 8 & 6 & \rngcell{11}{10} & 7 & 9 & \rngcell{9}{8} & \rngcell{12}{11} & \rngcell{12}{10} & \rngcell{16}{14} & 12 & 15 \\
41 & 6 & 8 & 7 & \rngcell{7}{6} & 9 & \rngcell{10}{9} & \rngcell{10}{9} & \rngcell{11}{10} & \rngcell{11}{9} & \rngcell{15}{12} & 11 & 9 \\
42 & \tikzmark{Lb42}\makebox[0pt]{7} & \tikzmark{Rb42}\makebox[0pt]{10} & 8 & 8 & 9 & \rngcell{16}{14} & 10 & 11 & \rngcell{11}{10} & \rngcell{15}{12} & \rngcell{11}{10} & 11 \\
43 & \tikzmark{Lb43}\makebox[0pt]{4} & \tikzmark{Rb43}\makebox[0pt]{5} & 4 & 5 & 5 & \rngcell{7}{6} & 6 & \rngcell{7}{6} & \rngcell{7}{5} & \rngcell{10}{7} & \rngcell{7}{6} & 8 \\
44 & \tikzmark{Lb44}\makebox[0pt]{3} & \tikzmark{Rb44}\makebox[0pt]{3} & 3 & 6 & 5 & \rngcell{7}{6} & 5 & \rngcell{11}{9} & \rngcell{6}{5} & \rngcell{8}{7} & 7 & 9 \\
45 & 7 & \tikzmark{Lb45}\makebox[0pt]{9} & \tikzmark{Rb45}\makebox[0pt]{8} & 9 & 10 & \rngcell{12}{11} & 11 & \rngcell{11}{10} & 13 & 15 & 14 & 15 \\
46 & \rngcell{13}{12} & \tikzmark{Lb46}\makebox[0pt]{\rngcell{21}{19}} & \tikzmark{Rb46}\makebox[0pt]{\rngcell{13}{11}} & \rngcell{12}{11} & \rngcell{13}{12} & \rngcell{16}{15} & 14 & 14 & 15 & \rngcell{23}{22} & 15 & 14 \\
47 & \rngcell{17}{16} & \tikzmark{Lb47}\makebox[0pt]{\rngcell{22}{20}} & \tikzmark{Rb47}\makebox[0pt]{\rngcell{17}{15}} & \rngcell{15}{14} & \rngcell{14}{13} & \rngcell{15}{14} & 14 & 15 & 15 & \rngcell{17}{16} & 15 & 12 \\
48 & 14 & \rngcell{20}{19} & \tikzmark{Lb48}\makebox[0pt]{14} & \tikzmark{Rb48}\makebox[0pt]{\rngcell{17}{16}} & \rngcell{9}{8} & 8 & 8 & 10 & 9 & 11 & 9 & \rngcell{13}{12} \\
49 & \rngcell{12}{9} & \rngcell{20}{15} & \tikzmark{Lb49}\makebox[0pt]{\rngcell{12}{9}} & \tikzmark{Rb49}\makebox[0pt]{\rngcell{10}{8}} & \rngcell{7}{6} & \rngcell{7}{6} & \rngcell{6}{5} & 5 & 5 & 7 & 5 & \rngcell{5}{4} \\
50 & \rngcell{12}{9} & \rngcell{18}{13} & \tikzmark{Lb50}\makebox[0pt]{\rngcell{12}{9}} & \tikzmark{Rb50}\makebox[0pt]{\rngcell{10}{9}} & 9 & \rngcell{16}{13} & \rngcell{8}{7} & 6 & 6 & 9 & 6 & \rngcell{6}{5} \\
\bottomrule
\end{tabular*}
\begin{tikzpicture}[overlay, remember picture]
  \coordinate (Pb42) at ($(pic cs:Lb42)!0.5!(pic cs:Rb42)+(\stairshift,\stairlift)$);
  \coordinate (Pb43) at ($(pic cs:Lb43)!0.5!(pic cs:Rb43)+(\stairshift,\stairlift)$);
  \coordinate (Pb44) at ($(pic cs:Lb44)!0.5!(pic cs:Rb44)+(\stairshift,\stairlift)$);
  \coordinate (Pb45) at ($(pic cs:Lb45)!0.5!(pic cs:Rb45)+(\stairshift,\stairlift)$);
  \coordinate (Pb46) at ($(pic cs:Lb46)!0.5!(pic cs:Rb46)+(\stairshift,\stairlift)$);
  \coordinate (Pb47) at ($(pic cs:Lb47)!0.5!(pic cs:Rb47)+(\stairshift,\stairlift)$);
  \coordinate (Pb48) at ($(pic cs:Lb48)!0.5!(pic cs:Rb48)+(\stairshift,\stairlift)$);
  \coordinate (Pb49) at ($(pic cs:Lb49)!0.5!(pic cs:Rb49)+(\stairshift,\stairlift)$);
  \coordinate (Pb50) at ($(pic cs:Lb50)!0.5!(pic cs:Rb50)+(\stairshift,\stairlift)$);
  \draw[CtpPink, line width=0.9pt, line join=miter, line cap=round] let \p1=(Pb42), \p2=(Pb43), \p3=(pic cs:Lb) in
    (\x3, {\y1+(\y1-\y2)/2}) -- (\x1, {\y1+(\y1-\y2)/2}) -- (\x1,\y1);
  \draw[CtpPink, line width=0.9pt, line join=miter, line cap=round] let \p1=(Pb42), \p2=(Pb43) in
    (\x1,\y1) -- (\x1,{(\y1+\y2)/2}) -- (\x2,{(\y1+\y2)/2}) -- (\x2,\y2);
  \draw[CtpPink, line width=0.9pt, line join=miter, line cap=round] let \p1=(Pb43), \p2=(Pb44) in
    (\x1,\y1) -- (\x1,{(\y1+\y2)/2}) -- (\x2,{(\y1+\y2)/2}) -- (\x2,\y2);
  \draw[CtpPink, line width=0.9pt, line join=miter, line cap=round] let \p1=(Pb44), \p2=(Pb45) in
    (\x1,\y1) -- (\x1,{(\y1+\y2)/2}) -- (\x2,{(\y1+\y2)/2}) -- (\x2,\y2);
  \draw[CtpPink, line width=0.9pt, line join=miter, line cap=round] let \p1=(Pb45), \p2=(Pb46) in
    (\x1,\y1) -- (\x1,{(\y1+\y2)/2}) -- (\x2,{(\y1+\y2)/2}) -- (\x2,\y2);
  \draw[CtpPink, line width=0.9pt, line join=miter, line cap=round] let \p1=(Pb46), \p2=(Pb47) in
    (\x1,\y1) -- (\x1,{(\y1+\y2)/2}) -- (\x2,{(\y1+\y2)/2}) -- (\x2,\y2);
  \draw[CtpPink, line width=0.9pt, line join=miter, line cap=round] let \p1=(Pb47), \p2=(Pb48) in
    (\x1,\y1) -- (\x1,{(\y1+\y2)/2}) -- (\x2,{(\y1+\y2)/2}) -- (\x2,\y2);
  \draw[CtpPink, line width=0.9pt, line join=miter, line cap=round] let \p1=(Pb48), \p2=(Pb49) in
    (\x1,\y1) -- (\x1,{(\y1+\y2)/2}) -- (\x2,{(\y1+\y2)/2}) -- (\x2,\y2);
  \draw[CtpPink, line width=0.9pt, line join=miter, line cap=round] let \p1=(Pb49), \p2=(Pb50) in
    (\x1,\y1) -- (\x1,{(\y1+\y2)/2}) -- (\x2,{(\y1+\y2)/2}) -- (\x2,\y2);
  \draw[CtpPink, line width=0.9pt, line join=miter, line cap=round] let \p1=(Pb49), \p2=(Pb50) in
    (\x2,\y2) -- (\x2, {\y2-(\y1-\y2)/2});
\end{tikzpicture}
\end{table}

\begin{table}[p]
\caption{Bounds on the order of the $2$-primary component of the $s$-stem of $S^n$, recorded as a base-$2$ logarithm (upper/lower), $n = 27, \dots, 38$.}
\label{tab:bounds-true-27-38}
\centering
\footnotesize
\setlength{\tabcolsep}{3pt}
\renewcommand{\arraystretch}{1.08}
\begin{tabular*}{\linewidth}{@{\extracolsep{\fill}} r!{\stairedge{Lc}\vrule}*{12}{c} @{\stairedge{Rc}}}
\toprule
$s \backslash n$ & 27 & 28 & 29 & 30 & 31 & 32 & 33 & 34 & 35 & 36 & 37 & 38 \\
\midrule
0 & \stabline & \stabline & \stabline & \stabline & \stabline & \stabline & \stabline & \stabline & \stabline & \stabline & \stabline & \stabline \\
1 & \stabline & \stabline & \stabline & \stabline & \stabline & \stabline & \stabline & \stabline & \stabline & \stabline & \stabline & \stabline \\
2 & \stabline & \stabline & \stabline & \stabline & \stabline & \stabline & \stabline & \stabline & \stabline & \stabline & \stabline & \stabline \\
3 & \stabline & \stabline & \stabline & \stabline & \stabline & \stabline & \stabline & \stabline & \stabline & \stabline & \stabline & \stabline \\
4 & \stabline & \stabline & \stabline & \stabline & \stabline & \stabline & \stabline & \stabline & \stabline & \stabline & \stabline & \stabline \\
5 & \stabline & \stabline & \stabline & \stabline & \stabline & \stabline & \stabline & \stabline & \stabline & \stabline & \stabline & \stabline \\
6 & \stabline & \stabline & \stabline & \stabline & \stabline & \stabline & \stabline & \stabline & \stabline & \stabline & \stabline & \stabline \\
7 & \stabline & \stabline & \stabline & \stabline & \stabline & \stabline & \stabline & \stabline & \stabline & \stabline & \stabline & \stabline \\
8 & \stabline & \stabline & \stabline & \stabline & \stabline & \stabline & \stabline & \stabline & \stabline & \stabline & \stabline & \stabline \\
9 & \stabline & \stabline & \stabline & \stabline & \stabline & \stabline & \stabline & \stabline & \stabline & \stabline & \stabline & \stabline \\
10 & \stabline & \stabline & \stabline & \stabline & \stabline & \stabline & \stabline & \stabline & \stabline & \stabline & \stabline & \stabline \\
11 & \stabline & \stabline & \stabline & \stabline & \stabline & \stabline & \stabline & \stabline & \stabline & \stabline & \stabline & \stabline \\
12 & \stabline & \stabline & \stabline & \stabline & \stabline & \stabline & \stabline & \stabline & \stabline & \stabline & \stabline & \stabline \\
13 & \stabline & \stabline & \stabline & \stabline & \stabline & \stabline & \stabline & \stabline & \stabline & \stabline & \stabline & \stabline \\
14 & \stabline & \stabline & \stabline & \stabline & \stabline & \stabline & \stabline & \stabline & \stabline & \stabline & \stabline & \stabline \\
15 & \stabline & \stabline & \stabline & \stabline & \stabline & \stabline & \stabline & \stabline & \stabline & \stabline & \stabline & \stabline \\
16 & \stabline & \stabline & \stabline & \stabline & \stabline & \stabline & \stabline & \stabline & \stabline & \stabline & \stabline & \stabline \\
17 & \stabline & \stabline & \stabline & \stabline & \stabline & \stabline & \stabline & \stabline & \stabline & \stabline & \stabline & \stabline \\
18 & \stabline & \stabline & \stabline & \stabline & \stabline & \stabline & \stabline & \stabline & \stabline & \stabline & \stabline & \stabline \\
19 & \stabline & \stabline & \stabline & \stabline & \stabline & \stabline & \stabline & \stabline & \stabline & \stabline & \stabline & \stabline \\
20 & \stabline & \stabline & \stabline & \stabline & \stabline & \stabline & \stabline & \stabline & \stabline & \stabline & \stabline & \stabline \\
21 & \stabline & \stabline & \stabline & \stabline & \stabline & \stabline & \stabline & \stabline & \stabline & \stabline & \stabline & \stabline \\
22 & \stabline & \stabline & \stabline & \stabline & \stabline & \stabline & \stabline & \stabline & \stabline & \stabline & \stabline & \stabline \\
23 & \stabline & \stabline & \stabline & \stabline & \stabline & \stabline & \stabline & \stabline & \stabline & \stabline & \stabline & \stabline \\
24 & \stabline & \stabline & \stabline & \stabline & \stabline & \stabline & \stabline & \stabline & \stabline & \stabline & \stabline & \stabline \\
25 & 2 & \stabline & \stabline & \stabline & \stabline & \stabline & \stabline & \stabline & \stabline & \stabline & \stabline & \stabline \\
26 & 3 & 2 & \stabline & \stabline & \stabline & \stabline & \stabline & \stabline & \stabline & \stabline & \stabline & \stabline \\
27 & 3 & \inftycell & 3 & \stabline & \stabline & \stabline & \stabline & \stabline & \stabline & \stabline & \stabline & \stabline \\
28 & 2 & 3 & 2 & 1 & \stabline & \stabline & \stabline & \stabline & \stabline & \stabline & \stabline & \stabline \\
29 & 2 & 2 & 1 & \inftycell & 0 & \stabline & \stabline & \stabline & \stabline & \stabline & \stabline & \stabline \\
30 & 5 & 7 & 4 & 3 & 2 & 1 & \stabline & \stabline & \stabline & \stabline & \stabline & \stabline \\
31 & 8 & 8 & 8 & 8 & 8 & \inftycell & 8 & \stabline & \stabline & \stabline & \stabline & \stabline \\
32 & 3 & 3 & 3 & 6 & 5 & 6 & 5 & 4 & \stabline & \stabline & \stabline & \stabline \\
33 & 5 & 5 & 5 & 5 & 6 & 7 & 6 & \inftycell & 5 & \stabline & \stabline & \stabline \\
34 & 8 & 11 & 8 & 8 & 8 & 11 & 8 & 7 & 6 & 5 & \stabline & \stabline \\
35 & 7 & 9 & 7 & 7 & 6 & 6 & 6 & 5 & 5 & \inftycell & 5 & \stabline \\
36 & 3 & 6 & 4 & 5 & 1 & 1 & 1 & 3 & 2 & 3 & 2 & 1 \\
37 & 4 & 5 & 5 & 4 & 3 & 4 & 3 & 3 & 4 & 4 & 3 & \inftycell \\
38 & \rngcell{7}{6} & \rngcell{10}{9} & 7 & 8 & 8 & 12 & 8 & 7 & 7 & 9 & 6 & 5 \\
39 & \rngcell{10}{9} & \rngcell{10}{9} & 10 & 10 & 12 & 14 & 12 & 11 & 10 & 10 & 10 & 9 \\
40 & 10 & 9 & 9 & 12 & 11 & 14 & 11 & 12 & 8 & 7 & 7 & 9 \\
41 & 9 & 10 & 8 & 8 & 9 & 10 & 9 & 7 & 6 & 5 & 5 & 5 \\
42 & 10 & \rngcell{16}{14} & 10 & 8 & 8 & 11 & 8 & 8 & 8 & \rngcell{11}{10} & 8 & 8 \\
43 & 8 & \rngcell{9}{7} & 7 & 5 & 3 & 3 & 3 & 3 & 5 & \rngcell{7}{6} & 5 & 5 \\
44 & 8 & 10 & 6 & 8 & 3 & 3 & 3 & 6 & 5 & 8 & 6 & 7 \\
45 & 13 & 15 & 11 & 9 & 9 & 11 & 9 & 9 & 10 & 11 & 11 & 10 \\
46 & 12 & 14 & 10 & \rngcell{11}{10} & \rngcell{11}{9} & \rngcell{17}{14} & \rngcell{11}{9} & \rngcell{9}{8} & \rngcell{9}{8} & \rngcell{12}{11} & 9 & 10 \\
47 & 11 & 13 & 11 & \rngcell{13}{12} & \rngcell{14}{12} & \rngcell{16}{13} & \rngcell{14}{12} & \rngcell{13}{12} & \rngcell{11}{10} & \rngcell{11}{10} & 11 & 11 \\
48 & 8 & \rngcell{10}{9} & 10 & 12 & \rngcell{12}{11} & \rngcell{16}{15} & 12 & 15 & 10 & 8 & 8 & 11 \\
49 & 7 & \rngcell{8}{6} & 8 & 9 & \rngcell{9}{8} & \rngcell{13}{11} & 9 & 7 & 7 & 7 & 5 & 5 \\
50 & 7 & \rngcell{15}{12} & 8 & \rngcell{9}{8} & \rngcell{9}{7} & \rngcell{13}{9} & \rngcell{9}{7} & \rngcell{9}{8} & \rngcell{8}{7} & \rngcell{14}{12} & 8 & 6 \\
\bottomrule
\end{tabular*}
\end{table}

\begin{table}[p]
\caption{Bounds on the order of the $2$-primary component of the $s$-stem of $S^n$, recorded as a base-$2$ logarithm (upper/lower), $n = 39, \dots, 50$.}
\label{tab:bounds-true-39-50}
\centering
\footnotesize
\setlength{\tabcolsep}{3pt}
\renewcommand{\arraystretch}{1.08}
\begin{tabular*}{\linewidth}{@{\extracolsep{\fill}} r!{\stairedge{Ld}\vrule}*{12}{c} @{\stairedge{Rd}}}
\toprule
$s \backslash n$ & 39 & 40 & 41 & 42 & 43 & 44 & 45 & 46 & 47 & 48 & 49 & 50 \\
\midrule
0 & \stabline & \stabline & \stabline & \stabline & \stabline & \stabline & \stabline & \stabline & \stabline & \stabline & \stabline & \stabline \\
1 & \stabline & \stabline & \stabline & \stabline & \stabline & \stabline & \stabline & \stabline & \stabline & \stabline & \stabline & \stabline \\
2 & \stabline & \stabline & \stabline & \stabline & \stabline & \stabline & \stabline & \stabline & \stabline & \stabline & \stabline & \stabline \\
3 & \stabline & \stabline & \stabline & \stabline & \stabline & \stabline & \stabline & \stabline & \stabline & \stabline & \stabline & \stabline \\
4 & \stabline & \stabline & \stabline & \stabline & \stabline & \stabline & \stabline & \stabline & \stabline & \stabline & \stabline & \stabline \\
5 & \stabline & \stabline & \stabline & \stabline & \stabline & \stabline & \stabline & \stabline & \stabline & \stabline & \stabline & \stabline \\
6 & \stabline & \stabline & \stabline & \stabline & \stabline & \stabline & \stabline & \stabline & \stabline & \stabline & \stabline & \stabline \\
7 & \stabline & \stabline & \stabline & \stabline & \stabline & \stabline & \stabline & \stabline & \stabline & \stabline & \stabline & \stabline \\
8 & \stabline & \stabline & \stabline & \stabline & \stabline & \stabline & \stabline & \stabline & \stabline & \stabline & \stabline & \stabline \\
9 & \stabline & \stabline & \stabline & \stabline & \stabline & \stabline & \stabline & \stabline & \stabline & \stabline & \stabline & \stabline \\
10 & \stabline & \stabline & \stabline & \stabline & \stabline & \stabline & \stabline & \stabline & \stabline & \stabline & \stabline & \stabline \\
11 & \stabline & \stabline & \stabline & \stabline & \stabline & \stabline & \stabline & \stabline & \stabline & \stabline & \stabline & \stabline \\
12 & \stabline & \stabline & \stabline & \stabline & \stabline & \stabline & \stabline & \stabline & \stabline & \stabline & \stabline & \stabline \\
13 & \stabline & \stabline & \stabline & \stabline & \stabline & \stabline & \stabline & \stabline & \stabline & \stabline & \stabline & \stabline \\
14 & \stabline & \stabline & \stabline & \stabline & \stabline & \stabline & \stabline & \stabline & \stabline & \stabline & \stabline & \stabline \\
15 & \stabline & \stabline & \stabline & \stabline & \stabline & \stabline & \stabline & \stabline & \stabline & \stabline & \stabline & \stabline \\
16 & \stabline & \stabline & \stabline & \stabline & \stabline & \stabline & \stabline & \stabline & \stabline & \stabline & \stabline & \stabline \\
17 & \stabline & \stabline & \stabline & \stabline & \stabline & \stabline & \stabline & \stabline & \stabline & \stabline & \stabline & \stabline \\
18 & \stabline & \stabline & \stabline & \stabline & \stabline & \stabline & \stabline & \stabline & \stabline & \stabline & \stabline & \stabline \\
19 & \stabline & \stabline & \stabline & \stabline & \stabline & \stabline & \stabline & \stabline & \stabline & \stabline & \stabline & \stabline \\
20 & \stabline & \stabline & \stabline & \stabline & \stabline & \stabline & \stabline & \stabline & \stabline & \stabline & \stabline & \stabline \\
21 & \stabline & \stabline & \stabline & \stabline & \stabline & \stabline & \stabline & \stabline & \stabline & \stabline & \stabline & \stabline \\
22 & \stabline & \stabline & \stabline & \stabline & \stabline & \stabline & \stabline & \stabline & \stabline & \stabline & \stabline & \stabline \\
23 & \stabline & \stabline & \stabline & \stabline & \stabline & \stabline & \stabline & \stabline & \stabline & \stabline & \stabline & \stabline \\
24 & \stabline & \stabline & \stabline & \stabline & \stabline & \stabline & \stabline & \stabline & \stabline & \stabline & \stabline & \stabline \\
25 & \stabline & \stabline & \stabline & \stabline & \stabline & \stabline & \stabline & \stabline & \stabline & \stabline & \stabline & \stabline \\
26 & \stabline & \stabline & \stabline & \stabline & \stabline & \stabline & \stabline & \stabline & \stabline & \stabline & \stabline & \stabline \\
27 & \stabline & \stabline & \stabline & \stabline & \stabline & \stabline & \stabline & \stabline & \stabline & \stabline & \stabline & \stabline \\
28 & \stabline & \stabline & \stabline & \stabline & \stabline & \stabline & \stabline & \stabline & \stabline & \stabline & \stabline & \stabline \\
29 & \stabline & \stabline & \stabline & \stabline & \stabline & \stabline & \stabline & \stabline & \stabline & \stabline & \stabline & \stabline \\
30 & \stabline & \stabline & \stabline & \stabline & \stabline & \stabline & \stabline & \stabline & \stabline & \stabline & \stabline & \stabline \\
31 & \stabline & \stabline & \stabline & \stabline & \stabline & \stabline & \stabline & \stabline & \stabline & \stabline & \stabline & \stabline \\
32 & \stabline & \stabline & \stabline & \stabline & \stabline & \stabline & \stabline & \stabline & \stabline & \stabline & \stabline & \stabline \\
33 & \stabline & \stabline & \stabline & \stabline & \stabline & \stabline & \stabline & \stabline & \stabline & \stabline & \stabline & \stabline \\
34 & \stabline & \stabline & \stabline & \stabline & \stabline & \stabline & \stabline & \stabline & \stabline & \stabline & \stabline & \stabline \\
35 & \stabline & \stabline & \stabline & \stabline & \stabline & \stabline & \stabline & \stabline & \stabline & \stabline & \stabline & \stabline \\
36 & \stabline & \stabline & \stabline & \stabline & \stabline & \stabline & \stabline & \stabline & \stabline & \stabline & \stabline & \stabline \\
37 & 2 & \stabline & \stabline & \stabline & \stabline & \stabline & \stabline & \stabline & \stabline & \stabline & \stabline & \stabline \\
38 & 4 & 3 & \stabline & \stabline & \stabline & \stabline & \stabline & \stabline & \stabline & \stabline & \stabline & \stabline \\
39 & 9 & \inftycell & 9 & \stabline & \stabline & \stabline & \stabline & \stabline & \stabline & \stabline & \stabline & \stabline \\
40 & 8 & 9 & 8 & 7 & \stabline & \stabline & \stabline & \stabline & \stabline & \stabline & \stabline & \stabline \\
41 & 6 & 7 & 6 & \inftycell & 5 & \stabline & \stabline & \stabline & \stabline & \stabline & \stabline & \stabline \\
42 & 8 & 11 & 8 & 7 & 6 & 5 & \stabline & \stabline & \stabline & \stabline & \stabline & \stabline \\
43 & 4 & 4 & 4 & 3 & 3 & \inftycell & 3 & \stabline & \stabline & \stabline & \stabline & \stabline \\
44 & 3 & 3 & 3 & 5 & 4 & 5 & 4 & 3 & \stabline & \stabline & \stabline & \stabline \\
45 & 9 & 9 & 8 & 8 & 9 & 9 & 8 & \inftycell & 7 & \stabline & \stabline & \stabline \\
46 & 10 & \rngcell{13}{12} & 9 & 8 & 8 & 10 & 7 & 6 & 5 & 4 & \stabline & \stabline \\
47 & 13 & \rngcell{15}{14} & 13 & 12 & 11 & 11 & 11 & 10 & 10 & \inftycell & 10 & \stabline \\
48 & 10 & 13 & 10 & 11 & 7 & 6 & 6 & 8 & 7 & 8 & 7 & 6 \\
49 & 6 & 7 & 6 & 4 & 3 & 2 & 2 & 2 & 3 & 4 & 3 & \inftycell \\
50 & 6 & 9 & 6 & 6 & 6 & \rngcell{9}{8} & 6 & 6 & 6 & 9 & 6 & 5 \\
\bottomrule
\end{tabular*}
\end{table}

\clearpage
\printbibliography
\end{document}